\documentclass[letterpaper,onecolumn]{IEEEtran}
\usepackage{graphicx}
\usepackage{amsmath,amssymb,amsthm} 
\usepackage{color}
\usepackage{times}

\newcommand{\ud}{\,\mathrm{d}}

\newcommand{\R}{\mathbb{R}}

\newcommand{\ip}[3]{\left< {#1}, {#2} \right>_{#3}}

\newcommand{\comment}[1]{ }
\newcommand{\der}[2]{\frac{\ud {#1}}{\ud {#2}}}
\newcommand{\pder}[2]{\frac{\partial {#1}}{\partial {#2}}}

\newcommand{\dv}[1]{\mbox{div}\left( {#1} \right)}

\newcommand{\tr}[1]{\mbox{tr}\left({#1}\right)}
\newcommand{\diff}{\mbox{Diff}(\R^n)}

\newtheorem{definition}{Definition}
\newtheorem{theorem}{Theorem}[section]
\newtheorem{lemma}{Lemma}
\newtheorem{remark}{Remark}

\begin{document}

\title{Accelerated Optimization in the PDE Framework: Formulations for
  the Manifold of Diffeomorphisms}

\author{Ganesh Sundaramoorthi\thanks{KAUST (King Abdullah University
    of Science and Technology), {\tt ganesh.sundaramoorthi@kaust.edu.sa } }
  and 
  Anthony Yezzi\thanks{Georgia Institute of Technology, {\tt ayezzi@ece.gatech.edu} } }

\maketitle

\begin{abstract}
  We consider the problem of optimization of cost functionals on the
  infinite-dimensional manifold of diffeomorphisms. We present a new
  class of optimization methods, valid for any optimization problem
  setup on the space of diffeomorphisms by generalizing Nesterov
  accelerated optimization to the manifold of diffeomorphisms. While
  our framework is general for infinite dimensional manifolds, we
  specifically treat the case of diffeomorphisms, motivated by optical
  flow problems in computer vision. This is accomplished by building
  on a recent variational approach to a general class of accelerated
  optimization methods by Wibisono, Wilson and Jordan
  \cite{wibisono2016variational}, which applies in finite dimensions. We
  generalize that approach to infinite dimensional manifolds. We
  derive the surprisingly simple continuum evolution equations, which
  are partial differential equations, for accelerated gradient
  descent, and relate it to simple mechanical principles from fluid
  mechanics. Our approach has natural connections to the optimal mass
  transport problem. This is because one can think of our approach as
  an evolution of an infinite number of particles endowed with mass
  (represented with a mass density) that moves in an energy
  landscape. The mass evolves with the optimization variable, and
  endows the particles with dynamics. This is different than the
  finite dimensional case where only a single particle moves and hence
  the dynamics does not depend on the mass. We derive the theory,
  compute the PDEs for accelerated optimization, and illustrate the
  behavior of these new accelerated optimization schemes.
\end{abstract}

\section{Introduction}

Accelerated optimization methods have gained wide applicability within
the machine learning and optimization communities (e.g.,
\cite{Bubeck15,Flammarion15,Ghadimi16,Hu09,Ji09,Jojic10,Krichene15,Li15,Nesterov05,Nesterov08,Nesterov83}). They
are known for leading to optimal convergence rates among schemes that
use only gradient (first order) information in the convex case.  In
the non-convex case, they appear to provide robustness to shallow
local minima.  The intuitive idea is that by considering a particle
with mass that moves in an energy landscape, the particle will gain
momentum and surpass shallow local minimum and settle in in more
wider, deeper local extrema in the energy landscape. This property has
made them (in conjunction with stochastic search algorithms)
particularly useful in machine learning, especially in the training of
deep networks, where the optimization is a non-convex problem that is
riddled with local minima. These methods have so far have only been
used in optimization problems that are defined in finite
dimensions. In this paper, we consider the generalization of these
methods to infinite dimensional manifolds. We are motivated by
applications in computer vision, in particular, segmentation, 3D
reconstruction, and optical flow. In these problems, the optimization
is over infinite dimensional geometric quantities (e.g., curves,
surfaces, mappings), and so the problems are formulated on infinite
dimensional manifolds. Recently there has been interest within the
machine learning community in optimization on finite dimensional
manifolds, such as matrix groups, e.g.,
\cite{zhang2016riemannian,liu2017accelerated,hosseini2017alternative},
which have particular structure not available on infinite dimensional
manifolds that we consider here.

\comment{
Although these techniques have been widely used, it was only within
the last few years that theoretical attempts have been made to
understand them and put them in a mathematical framework.  

Motivated by the success of accelerated methods in
finite dimensional problems, we formulate optimization problems on
infinite dimensional manifolds and, in particular, the infinite
dimensional group of diffeomorphisms using a generalization of this
approach.
}

Recent work \cite{wibisono2016variational} has shown that the
continuum limit of accelerated methods, which are discrete
optimization algorithms, may be formulated with variational
principles. This allows one to derive the continuum limit of
accelerated optimization methods (Nesterov's optimization method
\cite{Nesterov83} and others) as an optimization problem on descent
paths. The resulting optimal continuum path is defined by an ODE,
which when discretized appropriately yields Nesterov's method and
other accelerated optimization schemes. The optimization problem on
paths is an action integral, which integrates the Bregman
Lagrangian. The Bregman Lagrangian is a time-explicit Lagrangian (from
physics) that consists of kinetic and potential energies. The kinetic
energy is defined using the Bregman divergence (see
Section~\ref{sec:variational_accelerated}); it is designed for finite
step sizes, and thus differs from classical action integrals in
physics \cite{arnol2013mathematical,marsden2013introduction}. The
potential energy is the cost function that is to be optimized.

We build on the approach of \cite{wibisono2016variational} by
formulating accelerated optimization with an action integral, but we
generalize that approach to infinite dimensional manifolds. Our
approach is general for infinite dimensional manifolds, but we
illustrate the idea here for the case of the infinite dimensional
manifold of diffeomorphisms of $\R^n$ (the case of the manifold of
curves has been recently formulated by the authors
\cite{YezziSun2017}). To do this, we abandon the Bregman Lagrangian
framework in \cite{wibisono2016variational} since that assumes that
the variable over which one optimizes is embedded in $\R^n$.

Instead, we adopt the classical formulation of action integrals in
physics \cite{arnol2013mathematical,marsden2013introduction}, which is
already general enough to deal with manifolds, and kinetic energies
that are defined through general Riemannian metrics rather than a
traditional Euclidean metric, thus by-passing the need for the use of
Bregman distances. Our approach requires consideration of additional
technicalities beyond that of \cite{wibisono2016variational} and
classical physics. Namely, in finite dimensions in $\R^n$, one can
think of accelerated optimization as a single particle with mass
moving in an energy landscape. Since only a single particle moves,
mass is a fixed constant that does not impact the dynamics of the
particle. However, in infinite dimensions, one can instead think of an
infinite number of particles each moving, and these masses of
particles is better modeled with a \emph{mass density}. In the case of
the manifold of diffeomorphisms of $\R^n$ this mass density exists in
$\R^n$. As the diffeomorphism evolves to optimize the cost functional,
it deforms $\R^n$ and redistributes the mass, and so the density
changes in time. Since the mass density defines the kinetic energy and
the stationary action path depends on the kinetic energy, the dynamics
of the evolution to minimize the cost functional depends on the way
that mass is distributed in $\R^n$. Therefore, in the infinite
dimensional case, one also needs to optimize and account for the mass
density, which cannot be neglected. Further, our approach, due to the
infinite dimensional nature, has evolution equations that are PDEs
rather than ODEs in \cite{wibisono2016variational}. Finally, the
discretization of the resulting PDEs requires the use of entropy
schemes \cite{sethian1999level} since our evolution equations are
defined as viscosity solutions of PDEs, required to treat shocks and
rarefaction fans. These phenomena appear not to be present in the
finite dimensional case.

\subsection{Related Work}

\subsubsection{Sobolev Optimization}
Our work is motivated by Sobolev gradient descent approaches
\cite{Sundaramoorthi05,
  beg2005computing,charpiat2005designing,sundaramoorthi2007sobolev,charpiat2007generalized,sundaramoorthi2008coarse,sundaramoorthi2009new,mennucci2008sobolev,sundaramoorthi2011new,yang2015shape}
for optimization problems on manifolds, which have been used for
segmentation and optical flow problems. These approaches are general
in that they apply to non-convex problems, and they are derived by
computing the gradient of a cost functional with respect to a Sobolev
metric rather than an $L^2$ metric typically assumed in variational
optimization problems. The resulting gradient flows have been
demonstrated to yield coarse-to-fine evolutions, where the
optimization automatically transitions from coarse to successively
finer scale deformations. This makes the optimization robust to local
minimizers that plague $L^2$ gradient descents. We should point out
that the Sobolev metric is used beyond optimization problems and have
been used extensively in shape analysis (e.g.,
\cite{klassen2004analysis,michor2007metric,micheli2013sobolev,bauer2014overview}). While
such gradient descents are robust to local minimizers, computing them
in general involves an expensive computation of an inverse
differential operator at each iteration of the gradient descent. In
the case of optimization problems on curves and a very particular form
of a Sobolev metric this can be made computationally fast
\cite{sundaramoorthi2007sobolev}, but the idea does not generalize
beyond curves. In this work, we aim to obtain robustness properties of
Sobolev gradient flows, but without the expensive computation of
inverse operators.  Our accelerated approach involves averaging the
gradient across time in the descent process, rather than an averaging
across space in the Sobolev case. Despite our goal of avoiding Sobolev
gradients for computational speed, we should mention that our
framework is general to allow one to consider accelerated Sobolev
gradient descents (although we do not demonstrate it here), where
there is averaging in both space and time. This can be accomplished by
changing the definition of kinetic energy in our approach. This could
be useful in applications where added robustness is needed but speed
is not a critical factor.

\subsubsection{Optimal Mass Transport}
Our work relates to the literature on the problem of \emph{optimal
  mass transportion}
(e.g., \cite{villani2003topics,gangbo1996geometry,angenent2003minimizing}),
especially the formulation of the problem in
\cite{benamou2000computational}. The modern formulation of the
problem, called the Monge-Kantorovich problem, is as follows. One is
given two probability densities $\rho_0,\rho_1$ in $\R^n$, and the
goal is to compute a transformation $M : \R^n\to\R^n$ so that the
pushforward of $\rho_0$ by $M$ results in $\rho_1$ such that $M$ has
minimal cost. The cost is defined as the average Euclidean norm of
displacement: $\int_{\R^n} |M(x)-x|^p \rho_0(x)\ud x$ where $p\geq
1$. The value of the minimum cost is a distance (called the $L^p$
Wasserstein distance) on the space of probability measures.  In the
case that $p=2$, \cite{benamou2000computational} has shown that mass
transport can be formulated as a fluid mechanics problem. In
particular, the Wasserstein distance can be formulated as a distance
arising from a Riemannian metric on the space of probability
densities. The cost can be shown equivalent to the minimum Riemannian
path length on the space of probability densities, with the initial
and final points on the path being the two densities $\rho_0,
\rho_1$. The tangent space is defined to be velocities of the density
that infinitesimally displace the density. The Riemannian metric is
just the kinetic energy of the mass distribution as it is displaced by
the velocity, given by $\int_{\R^n} \frac 1 2 \rho(x) |v(x)|^2 \ud
x$. Therefore, optimal mass transport computes an optimal \emph{path}
on densities that minimizes the integral of kinetic energy along the
path.

In our work, we seek to minimize a potential on the space of
diffeomorphisms, with the use of acceleration. We can imagine that
each diffeomorphism is associated with a point on a manifold, and the
goal is to move to the bottom of the potential well. To do so, we
associate a mass density in $\R^n$, which as we optimize the
potential, moves in $\R^n$ via a push-forward of the evolving
diffeomorphism. We regard this evolution as a path in the space of
diffeomorphisms that arises from an action integral, where the action
is the difference of the kinetic and potential energies. The kinetic
energy that we choose, purely to endow the diffeomorphism with
acceleration, is the same one used in the fluid mechanics formulation
of optimal mass transportation for $p=2$. We have chosen this kinetic
energy for simplicity to illustrate our method, but we envision a
variety of kinetic energies can be defined to introduce different
dynamics. The main difference of our approach to the fluid mechanics
formulation of mass transport is in the fact that we do not minimize
just the path integral of the kinetic energy, but rather we derive our
method by computing stationary paths of the path integral of kinetic
minus \emph{potential} energies. Since diffeomorphisms are generated
by smooth velocity fields, we equivalently optimize over
velocities. We also optimize over the mass distribution. Thus, the
main difference between the fluid mechanics formulations of $L^2$ mass
transport and our approach is the potential on diffeomorphisms, which
is used to define the action integral.

\subsubsection{Diffeomorphic Image Registration} 
Our work relates to the literature on diffeomorphic image registration
\cite{beg2005computing,miller2006geodesic}, where the goal, similar to
ours, is to compute a registration between two images as a
diffeomorphism. There a diffeomorphism is generated by a path of
smooth velocity fields integrated over the path. Rather than
formulating an optimization problem directly on the diffeomorphism,
the optimization problem is formed on a path of velocity fields. The
optimization problem is to minimize $\int_0^1 \| v\|^2 \ud t$ where $v$
is a time varying vector field, $\|\cdot\|$ is a norm on velocity
fields, and the optimization is subject to the constraint that the
mapping $\phi$ maps one image to the other, i.e.,
$I_1 = I_0\circ\phi^{-1}$. The minimization can be considered the
minimization of an action integral where the action contains only a
kinetic energy. The norm is chosen to be a Sobolev norm to ensure that
the generated diffeomorphism (by integrating the velocity fields over
time) is smooth. The optimization problem is solved in
\cite{beg2005computing} by a Sobolev gradient descent on the
\emph{space of paths}. The resulting path is a geodesic with
Riemannian metric given by the Sobolev metric $\|v\|$. In
\cite{miller2006geodesic}, it is shown these geodesics can be computed
by integrating a forward evolution equation, determined from the
conservation of momentum, with an initial velocity.

Our framework instead uses accelerated gradient descent. Like
\cite{beg2005computing,miller2006geodesic}, it is derived from an
action integral, but the action has both a kinetic energy and a
\emph{potential} energy, which is the objective functional that is to
be optimized. In this current work, our kinetic energy arises
naturally from physics rather than a Sobolev norm. One of our
motivations in this work is to get regularizing effects of Sobolev
norms without using Sobolev norms, since that requires inverting
differential operators in the optimization, which is computationally
expensive. Our kinetic energy is an $L^2$ metric weighted by
\emph{mass}. Our method has acceleration, rather than zero
acceleration in \cite{beg2005computing,miller2006geodesic}, and this
is obtained by endowing a diffeomorphism with mass, which is a mass
density in $\R^n$. This mass allows for the kinetic energy to endow
the optimization with dynamics. Our optimization is obtained as the
stationary conditions of the action with respect to both velocity and
\emph{mass density}. The latter links our approach to optimal mass
transport, described earlier. Our physically motivated kinetic energy
and in particular the mass consideration allows us to generate
diffeomorphisms without the use fo Sobolev norms. We also avoid the
inversion of differential operators.

\subsubsection{Optical Flow}
Although our framework is general in solving any optimization on
infinite dimensional manifolds, we demonstrate the framework for
optimization of diffeomorphisms and specifically for optical flow
problems formulated as variational problems in computer vision (e.g.,
\cite{horn1981determining,black1996robust,brox2004high,wedel2009improved,sun2010secrets,yang2013modeling,yang2015self}). Optical
flow, i.e., determining pixel-wise correspondence between images, is a
fundamental problem in computer vision that remains a challenge to
solve, mainly because optical flow is a non-convex optimization
problem, and thus few methods exist to optimize such problems. Optical
flow was first formulated as a variational problem in
\cite{horn1981determining}, which consisted of a data fidelity term
and regularization favoring smooth optical flow. Since the problem is
non-convex, approaches to solve this problem typically involve the
assumption of small displacement between frames, so a linearization of
the data fidelity term can be performed, and this results in a problem
in which the global optimum of \cite{horn1981determining} can be
solved via the solution of a linear PDE. Although standard gradient
descent could be used on the non-linearized problem, it is numerically
sensitive, extremely computationally costly, and does not produce
meaningful results unless coupled with the strategy described next.
Large displacements are treated with two strategies: iterative warping
and image pyramids. Iterative warping involves iteration of the
linearization around the current accumulated optical flow. By use of
image pyramids, a large displacement is converted to a smaller
displacement in the downsampled images. While this strategy is
successful in many cases, there are also many problems associated with
linearization and pyramids, such as computing optical flow of thin
structures that undergo large displacements.  This basic strategy of
linearization, iterative warping and image pyramids have been the
dominant approach to many variational optical flow models (e.g.,
\cite{horn1981determining,black1996robust,brox2004high,wedel2009improved,sun2010secrets}),
regardless of the regularization that is used (e.g., use of robust
norms, total variation, non-local norms, etc). In
\cite{wedel2009improved}, the linearized problem with TV
regularization has been formulated as a convex optimization problem,
in which a primal-dual algorithm can be used. In \cite{yang2015shape}
linearization is avoided and rather a gradient descent with respect to
a Sobolev metric is computed, and is shown to have a automatic
coarse-to-fine optimization behavior. Despite these works, most
optical flow algorithms involve simplification of the problem into a
linear problem. In this work, we construct accelerated gradient
descent algorithms that are applicable to any variational optical flow
algorithm in which we avoid the linearization step and aim to obtain a
better optimizer. For illustration, we consider here the case of
optical flow modeled as a global diffeomorphism, but in principle this
can be generalized to piecewise diffeomorphisms as in
\cite{yang2015self}. Since diffeomorphisms do not form a linear space,
rather a infinite-dimensional manifold, we generalize accelerated
optimization to that space.

\section{Background for Accelerated Optimization on Manifolds}

\subsection{Manifolds and Mechanics}

We briefly summarize the key facts in classical mechanics that are the
basis for our accelerated optimization method on manifolds.

\subsubsection{Differential Geometry}
We review differential geometry (from \cite{do1992riemannian}), as
this will be needed to derive our accelerated optimization scheme on
the \emph{manifold} of diffeomorphisms. First a \emph{manifold} $M$ is
a space in which every point $p\in M$ has a (invertible) mapping $f_p$
from a neighborhood of $p$ to a \emph{model space} that is a linear
normed vector space, and has an additional compatibility condition
that if the neighborhoods for $p$ and $q$ overlap then the mapping
$f_p\circ f_{q}^{-1}$ is differentiable. Intuitively, a manifold is a
space that locally appears flat. The model space may be finite or
infinite dimensional when the model spaces are finite or infinite
dimensional, respectively. In the latter case the manifold is referred
to as an \emph{infinite dimensional manifold} and in the former case a
\emph{finite dimensional manifold}. The space of diffeomorphisms of
$\R^n$, the space of interest in this paper, is an infinite
dimensional manifold. The \emph{tangent space} at a point $p\in M$ is
the equivalence class, $[\gamma]$, of curves $\gamma : [0,1] \to M$
under the equivalence that $\gamma(0)=p$ and $(f_p\circ \gamma)'(0)$
are the same for each curve $\gamma\in [\gamma]$. Intuitively, these
are the set of possible directions of movement at the point $p$ on the
manifold. The \emph{tangent bundle}, denoted $TM$, is
$TM = \{ (p,v) \,:\, p\in M, v \in T_pM\}$, i.e., the space formed
from the collection of all points and tangent spaces.

In this paper, we will assume additional structure on the manifold,
namely, that an inner product (called the \emph{metric}) exists on
each tangent space $T_pM$. Such a manifold is called a
\emph{Riemannian manifold}. A Riemannian manifold allows one to
formally define the lengths of curves $\gamma : [-1,1]\to M$ on the
manifold. This allows one to construct paths of critical length,
called \emph{geodesics}, a generalization of a path on constant
velocity on the manifold. Note that while existence of geodesics is
guaranteed on finite dimensional manifolds, in the infinite
dimensional case, there is no such guarantee. The Riemannian metric
also allows one to define \emph{gradients} of functions $g : M \to \R$
defined on the manifold: the gradient $\nabla g(p) \in T_pM$ is
defined to be the vector that satisfies
$\der{}{\varepsilon} \left. g( \gamma(\varepsilon) ) \right|_{
  \varepsilon = 0 } = \ip{ \nabla g(p) }{ \gamma'(0) }{} $, where
$\gamma(0)=p$, the left hand side is the directional derivative and
the right hand side is the inner product from the Riemannian
structure.

\subsubsection{Mechanics on Manifolds}
We now briefly review some of the formalism of classical mechanics on
manifolds that will be used in this paper. The material is from
\cite{arnol2013mathematical,marsden2013introduction}. The subject of
mechanics describes the principles governing the evolution of a
particle that moves on a manifold $M$. The equations governing a
particle are Newton's laws. There are two viewpoints in mechanics,
namely the \emph{Lagrangian} and \emph{Hamiltonian} viewpoints, which
formulate more general principles to derive Newton's equations. In
this paper, we use the Lagrangian formulation to derive equations of
motion for accelerated optimization on the manifold of
diffeomorphisms. Lagrangian mechanics obtains equations of motion
through \emph{variational principles}, which makes it easier to
generalize Newton's laws beyond simple particle systems in $\R^3$,
especially to the case of manifolds. In Lagrangian mechanics, we start
with a function $L : TM \to \R$, called the Lagrangian, from the
tangent bundle to the reals. Here we assume that $M$ is a Riemannian
manifold. One says that a curve $\gamma : [-1,1] \to M$ is \emph{a
  motion in a Lagrangian system} with Lagrangian $L$ if it is an
extremal of $A = \int L(\gamma(t), \dot{\gamma}(t)) \ud t$. The
previous integral is called an \emph{action
  integral}. \emph{Hamilton's principle of stationary action} states
that the motion in the Lagrangian system satisfies the condition that
$\delta A = 0$, where $\delta$ denotes the variation, for \emph{all}
variations of $A$ induced by variations of the path $\gamma$ that keep
endpoints fixed. The variation is defined as
$\delta A := \der{}{s} \left. A( \tilde \gamma(t,s) ) \right|_{s=0}$
where $\tilde \gamma : [-1,1]^2 \to M$ is a smooth family of curves (a
variation of $\gamma$) on the manifold such that
$\tilde\gamma(t,0) = \gamma(t)$. The stationary conditions give rise
to what is known as \emph{Lagrange's} equations. A \emph{natural
  Lagrangian} has the special form $L = T - U$ where
$T : TM \to \R^{+}$ is the \emph{kinetic energy} and $U : M \to \R$ is
the \emph{potential energy}. The kinetic energy is defined as
$T(v) = \frac 1 2 \ip{v}{v}{} $ where $\ip{\cdot}{\cdot}{}$ is the
inner product from the Riemannian structure. In the case that one has
a particle system in $\R^3$, i.e., a collection of particles with
masses $m_i$, in a natural Lagrangian system, one can show that
Hamilton's principle of stationary action is equivalent to Newton's
law of motion, i.e., that $\der{}{t} (m_i \dot r_i) = -\pder{U}{r_i} $
where $r_i$ is the trajectory of the $i^{\text{th}}$ particle, and
$\dot{r}_i$ is the velocity. This states that mass times acceleration
is the force, which is given by minus the derivative of the potential
in a conservative system. Thus, Hamilton's principle is more general
and allows us to more easily derive equations of motion for more
general systems, in particular those on manifolds.

In this paper, we will consider \emph{Lagrangian non-autonomous
  systems} where the Lagrangian is also an explicit function of time
$t$, i.e., $L : TM \times \R \to \R$. In particular, the kinetic and
potential energies can both be explicit functions of time:
$T : TM\times\R \to \R$ and $U : M \times \R \to \R$. Autonomous
systems have an \emph{energy conservation property} and do not
converge; for instance, one can think of a moving pendulum with no
friction, which oscillates forever. Since the objective in this paper
is to minimize an objective functional, we want the system to
eventually converge and Lagrangian non-autonomous systems allow for
this possibility. For completness, we present some basic facts of the
Hamiltonian perspective to elaborate on the previous point, although
we do not use this in the present paper. The generalization of total
energy is the \emph{Hamiltonian}, defined as the Legendre transform of
the Lagrangian: $H(p,q,t) = \ip{p}{\dot{q}}{} - L(q,\dot{q}, t)$ where
$p = \der{L}{\dot{q}}$ is the fiber derivative of $L$ with respect to
$\dot{q}$, i.e.,
$\der{L}{\dot{q}} \cdot w = \der{}{\varepsilon} \left. L(q, \dot q +
  \varepsilon w ) \right|_{\varepsilon = 0}$. From the Hamiltonian,
one can also obtain a system of equations describing motion on the
manifold. It can be shown that if $L=T-U$ then $H=T+U$ and more
generally, $\der{H}{t} = -\pder{L}{t}$ along the stationary path of
the action. Thus, if the Lagrangian is natural and autonomous, the
total energy is preserved, otherwise energy could be dissipated based
on the partial of the Lagrangian with respect to $t$.

\subsection{Variational Approach to Accelerated Optimization in Finite
  Dimensional Vector Spaces}
\label{sec:variational_accelerated}

Accelerated gradient optimization can be motivated by the desire to
make an ordinary gradient descent algorithm 1) more robust to noise
and local minimizers, and 2) speed-up the convergence while only using
first order (gradient) information. For instance, if one computes a
noisy gradient due imperfections in obtaining an accurate gradient, a
simple heuristic to make the algorithm more robust is to compute a
running average of the gradient over iterations, and use that as the
search direction. This also has the advantage, for instance in
speeding up optimization in narrow shallow valleys. Gradient descent
(with finite step sizes) would bounce back and forth across the valley
and slowly descend down, but averaging the gradient could cancel the
component across the valley and more quickly optimize the
function. Strategic dynamically changing weights on previous gradients can
boost the descent rate. Nesterov put forth the following famous scheme
\cite{Nesterov83} which attains an optimal rate of order $\frac{1}{t^{2}}$
in the case of a smooth, convex cost function $f(x)$:
\[
y_{k+1}=x_{k}-\frac{1}{\beta}\nabla f(x_{k}),\qquad x_{k+1}=(1-\gamma_{k})y_{k+1}+\gamma_{k}y_{k},\qquad\gamma_{k}=\frac{1-\lambda_{k}}{\lambda_{k}+1},\qquad\lambda_{k}=\frac{1+\sqrt{1+4\lambda_{k-1}^{2}}}{2}
\]
where $x_{k}$ is the $k$-th iterate of the algorithm, $y_{k}$ is an
intermediate sequence, and $\gamma_{k}$ are dynamically updated
weights.

Recently \cite{wibisono2016variational} presented a variational
generalization of Nesterov's \cite{Nesterov83} and other accelerated
gradient descent schemes in $\mathbb{R}^{n}$ based on the Bregman
divergence of a convex distance generating function $h$:
\begin{equation}
  d(y,x)=h(y)-h(x)- \nabla h(x)\cdot (y-x)  \label{eq:breg-divergence}
\end{equation}
and careful discretization of the Euler-Lagrange equations for the
time integral of the following Bregman Lagrangian
\[
  L(X,V ,t)=e^{a(t)+\gamma(t)}\left[d(X+e^{-a(t)}V,X)-e^{b(t)} U(X)\right]
\]
where the potential energy $U$ represents the cost to be minimized.
In the Euclidean case where $h(x)=\frac 1 2 |x|^2$ gives
$d(y,x)=\frac{1}{2}|y-x|^{2}$, this simplifies to
\[
  L=e^{\gamma(t)}\left[ e^{-a(t)} \frac{1}{2}|V|^{2}-e^{a(t)+b(t)}U(X)\right]
\]
where $T=\frac 1 2 |V|^2$ is the kinetic energy of a unit mass particle
in $\mathbb{R}^{n}$. Nesterov's methods \cite{Nesterov83,Nesterov14,Nesterov13,Nesterov08,Nesterov06,Nesterov05}
belong to a subfamily of Bregman Lagrangians with the following choice
of parameters (indexed by $k>0$) 
\[
a=\log k-\log t,\qquad b=k\,\log t+\log\lambda,\qquad\gamma=k\,\log t
\]
which, in the Euclidean case, yields a non-autonomous Lagrangian as follows:
\begin{equation}
  L=\frac{t^{k+1}}{k}\left( T-\lambda k^{2}t^{k-2} U \right)\label{eq:time-action}
\end{equation}
In the case of $k=2$, for example, the stationary conditions of the
integral of this time-explicit action yield the continuum limit of
Nesterov's accelerated mirror descent \cite{Nesterov05} derived in
both \cite{Su14,Krichene15}.

Since the Bregman Lagrangian assumes that the underlying manifold is a
subset of $\R^n$ (in order to define the Bregman distance\footnote{One
  could in fact generalize such operations as addition and subtraction
  in manifolds, using the exponential and logarithmic maps. We avoid
  this since in the types of manifolds that we deal with, computing
  such maps itself requires solving a PDE or another optimization
  problem. We avoid all these complications, by going back to the
  formalism in classical mechanics.}), which many manifolds do not
have - for instance the manifold of diffeomorphisms that we consider
in this paper, we instead use the original classical mechanics
formulation, which already provides a formalism for considering
general metrics though the Riemannian distance, although not
equivalent to the Bregman distance.

\section{Accelerated Optimization for Diffeomorphisms}

In this section, we use the mechanics of particles on manifolds
developed in the previous section, and apply it to the case of the
infinite-dimensional manifold of diffeomorphisms in $\R^n$ for general
$n$. This allows us to generalize accelerated optimization to infinite
dimensional manifolds. Diffeomorphisms are smooth mappings
$\phi : \R^n \to \R^n$ whose inverse exists and is also smooth.
Diffeomorphisms form a group under composition. The inverse operator
on the group is defined as the inverse of the function, i.e.,
$\phi^{-1}(\phi(x)) = x$. Here smoothness will mean that two
derivatives of the mapping exist. The group of diffeomorphisms will be
denoted $\mbox{Diff}(\R^n)$.  Diffeomorphisms relate to image
registration and optical flow, where the mappings between two images
are often modeled as diffeomorphisms\footnote{In medical imaging, the
  model of diffeomorphisms for registration is fairly accurate since
  typically full 3D scans are available and thus all points in one
  image correspond to the other image and vice versa. Of course there
  are situations (such as growth of tumors) where the diffeomorphic
  assumption is invalid. In vision, typically images have occlusion
  phenomena and multiple objects moving in different ways. So a
  diffeomorphism is not a valid assumption, it is however a good model
  when restricted to a single object in the un-occluded
  part.}. Recovering diffeomorphisms from two images will be
formulated as an optimization problem $U(\phi)$ where $U$ will
correspond to the potential energy. Note we avoid calling $U$ the
energy as is customary in computer vision literature, because for us
the energy will refer to the total mechanical energy (i.e., the sum of
the kinetic and potential energies).  We do not make any assumptions
on the particular form of the potential in this section, as our goal
is to be able to accelerate \emph{any} optimization problem for
diffeomorphisms, given that one can compute a gradient of the
potential. The formulation here allows any of the numerous cost
functionals developed over the past three decades for image
registration to be accelerated.

In the first sub-section, we give the formulation and evolution
equations for the case of acceleration without energy dissipation
(Hamiltonian is conserved), since most of the calculations are
relevant for the case of energy dissipation, which is needed for the
evolution to converge to a diffeomorphism. In the second sub-section,
we formulate and compute the evolution equations for the energy
dissipation case, which generalizes Nesterov's method to the infinite
dimensional manifold of diffeomorphisms. Finally, in the last
sub-section we give an example potential and its gradient calculation
for a standard image registration or optical flow problem.

\subsection{Acceleration Without Energy Dissipation}

\subsubsection{Formulation of the Action Integral}
Since the potential energy $U$ is assumed given, in order to formulate
the action integral in the non-dissipative case, we need to define
kinetic energy $T$ on the space of diffeomorphisms. Since
diffeomorphisms form a manifold, we can apply the the results in the
previous section and note that the kinetic energy will be defined on
the tangent space to $\mbox{Diff}(\R^n)$ at a particular
diffeomorphism $\phi$. This will be denoted
$T_{\phi}\mbox{Diff}(\R^n)$. The tangent space at $\phi$ can be
roughly thought of as the set of local perturbations $v$ of $\phi$
given for all $\varepsilon$ small that perserve the diffeomorphism
property, i.e., $\phi + \varepsilon v$ is a diffeomorphism. One can
show that the tangent space is given by
\begin{equation}
  T_{\phi} \mbox{Diff}(\R^n) = \{  v : \phi(\R^n) \to \R^n \,:\, v \mbox{ is
  smooth } \}.
\end{equation}
In the above, since $\phi$ is a diffeomorphism, we have that
$\phi(\R^n) = \R^n$. However, we write $v : \phi(\R^n) \to \R^n$ to
emphasize that the velocity fields in the tangent space are defined on
the range of $\phi$, so that $v$ is interpreted as a Eulerian
velocity. By definition of the tangent space, an infinitesimal
perturbation of $\phi$ by a tangent vector, given by
$\phi + \varepsilon v$, will be a diffeomorphism for $\varepsilon$
sufficiently small. Note that the previous operation of addition is
defined as follows:
\[
  ( \phi + \varepsilon v )(x) = \phi(x) + \varepsilon v(\phi(x)).
\]
The tangent space is a set of smooth vector fields on $\phi(\R^n)$ in
which the vector field at each point $\phi(x)$, displaces $\phi(x)$
infinitesimally by $v(\phi(x))$ to form another diffeomorphism.

We note a classical result from \cite{ebin1970groups}, which will be
of utmost importance in our derivation of accelerated optimization on
$\diff$. The result is that any (orientable) diffeomorphism may be
generated by integrating a time-varying smooth vector field over time,
i.e.,
\begin{equation} \label{eq:phi_evol}
  \partial_t \phi_t(x) = v_t( \phi_t(x) ), \quad x\in \R^n,
\end{equation}
where $\partial_t$ denotes partial derivative with respect to $t$,
$\phi_t$ denotes a time varying family of diffeomorphisms evaluated at
the time $t$, and $v_t$ is a time varying collection of vector fields
evaluated at time $t$. The path $t \to \phi_t(x)$ for a fixed $x$
represents a trajectory of a particle starting at $x$ and flowing
according to the velocity field.

The space on which the kinetic energy is defined is now clear, but one
more ingredient is needed before we can define the kinetic energy. Any
accelerated method will need a notion of \emph{mass}, otherwise
acceleration is not possible, e.g., a mass-less ball will not
accelerate. We generalize the concept of mass to the infinite
dimensional manifold of diffeomorphisms, where there are infinitely
more possibilities than a single particle in the finite dimensional
case considered by \cite{wibisono2016variational}. There optimization
is done on a finite dimensional space, the space of a \emph{single}
particle, and the possible choices of mass are just different fixed
constants. The choice of the constant, given the particle's mass
remains fixed, is irrelevant to the final evolution. This is different
in than the case of diffeomorphisms. Here we imagine that an infinite
number of particles densely distributed in $\R^n$ with mass exist and
are displaced by the velocity field $v$ at every point. Since the
particles are densely distributed, it is natural to represent the mass
of all particles with a \emph{mass density} $\rho : \R^n \to \R$,
similar to a fluid at a fixed time instant. The density $\rho$ is
defined as mass divided by volume as the volume shrinks.  During the
evolution to optimize the potential $U$, the particles are displaced
continuously and thus the density of these particles will in general
change over time. Note the density will change even if the density at
the start is constant except in the case of full translation motion
(when $v$ is spatially constant). The latter case is not general
enough, as we want to capture general diffeomorphisms. We will assume
that the system of particles in $\R^n$ is closed and so we impose a
\emph{mass preservation constraint}, i.e.,
\begin{equation} \label{eq:mass_conserve}
  \int_{\R^n} \rho(x) \ud x = 1,
\end{equation}
where we assume the total mass is one without loss of generality. Note
that the evolution of a time varying density $\rho_t$ as it is
deformed in time by a time varying velocity is given by the
\emph{continuity equation}, which is a local form of the conservation
of mass given by \eqref{eq:mass_conserve}. The continuity equation is
defined by the partial differential equation
\begin{equation} \label{eq:continuity_eqn}
  \partial_t \rho(x) + \dv{\rho(x) v(x)} = 0, \quad x\in \R^n
\end{equation}
where $\dv{}$ denotes the divergence operator acting on a vector
field and is $\dv{F} = \sum_{i=}^n \partial_{x_i} F^i$ where
$\partial_{x_i}$ is the partial with respect to the $i^{\text{th}}$
coordinate and $F^i$ is the $i^{\text{th}}$ component of the vector
field. We will assume that the mass distribution dies down to zero
outside a compact set so as to avoid boundary considerations in our derivations.

We now have the two ingredients, namely the tangent vectors to $\diff$
and the concept of mass, which allows us to define a natural physical
extension of the kinetic energy to the case of an infinite mass
distribution. We present one possible kinetic energy to illustrate the
idea of accelerated optimization, but this is by no means the only
definition of kinetic energy. We envision this to be part of the
design process in which one could get a multitude of various different
accelerated optimization schemes by defining different kinetic
energies. Our definition of kinetic energy is just the kinetic energy
arising from fluid mechanics:
\begin{equation}
  T(v) = \int_{ \phi(\R^n) } \frac 1 2 \rho(x) |v(x)|^2 \ud x,
\end{equation}
which is just the integration of single particle's kinetic energy
$\frac 1 2 m |v|^2$ and matches the definition of the kinetic energy
of a sum of particles in elementary physics. Note that the kinetic
energy is just one-half times the norm squared for the norm arising
from the Riemannian metric \cite{arnol2013mathematical}, i.e., an
inner product on the tangent space of $\diff$. The Riemannian metric
is given by
$\ip{v_1}{v_2}{} = \int_{\R^n} \rho(x) v_1(x) \cdot v_2(x) \ud x$,
which is just a weighted $\mathbb{L}^2$ inner product.

We are now ready to define the action integral for the case of
$\diff$, which is defined on \emph{paths} of diffeomorphisms. A path
of diffeomorphisms is $\phi : [0,\infty) \times \R^n \to \R^n$ and we
will denote the diffeomorphism at time $t$ along this path as
$\phi_t$.  Since diffeomorphisms are generated by velocity fields, we
may equivalently define the action in terms of \emph{paths} of
velocity fields. A path of velocity fields is given by
$v : [0,\infty) \times \R^n \to \R^n$, and the velocity at time $t$
along the path is denoted $v_t$. Notice that the action requires a
kinetic energy and the kinetic energy is dependent on the mass
density. Thus, a path of densities
$\rho : [0,\infty) \times \R^n \to \R^+$ is required, which represents
the mass distribution of the particles in $\R^n$ as they are deformed
along time by the velocity field $v_t$. This path of densities is
subject to the continuity equation. With this, the action integral
is then 
\begin{equation} \label{eq:action}
  A = \int \left[ T(v_t) - U(\phi_t) \right] \ud t,
\end{equation}
where the integral is over time, and we do not specify the limits of
integration as it is irrelevant as the endpoints will be fixed and the
action will be thus independent of the limits. Note that the action is
implicitly a function of three paths, i.e., $v_t,\phi_t$ and
$\rho_t$. Further, these paths are not independent of each other as
$\phi_t$ depends on $v_t$ through the generator relation
\eqref{eq:phi_evol}, and $\rho_t$ depends on $v_t$ through the
continuity equation \eqref{eq:continuity_eqn}.

\subsubsection{Stationary Conditions for the Action}
We now derive the stationary conditions for the action integral
\eqref{eq:action}, and thus the evolution equation for a path of
diffeomorphisms, which is Hamilton's principle of stationary action,
equivalent to a generalization of Newton's laws of motion extended to
diffeomorphisms. As discussed earlier, we would like to find the
stationary conditions for the action integral \eqref{eq:action},
defined on the path $\phi_t$, under the conditions that it is
generated by a path of smooth velocity fields $v_t$, which is also
coupled with the mass density $\rho_t$.

We treat the computation of the stationary conditions of the action as
a constrained optimization problem with respect to the two
aforementioned constraints. To do this, it is easier to formulate the
action in terms of the path of the inverse diffeomorphisms
$\phi^{-1}_t$, which we will call $\psi_t$. This is because the
non-linear PDE constraint \eqref{eq:phi_evol} can be equivalently
reformulated as the following linear transport PDE in the inverse
mappings:
\begin{equation} \label{eq:transport_backward}
  \partial_t \psi_t(x) + [D\psi_t(x)]v_t(x) = 0, \quad x\in \R^n
\end{equation}
where $D$ denotes the derivative (Jacobian) operator. To derive the
stationary conditions with respect to the constraints, we use the
method of Lagrange multipliers. We denote by
$\lambda : [0,\infty) \times \R^n \to \R^n$ the Lagrange multiplier
according to \eqref{eq:transport_backward}. We denote
$\mu : [0,\infty) \times \R^n \to \R$ as the Lagrange multiplier for
the continuity equation \eqref{eq:continuity_eqn}. Because we would
like to be able to have possibly discontinuous solutions of the
continuity equation, we formulate it in its weak form by multiplying
the constraint by the Lagrange multiplier and integrating by parts
thereby removing the derivatives on possibly discontinuous $\rho$:
\begin{equation}
  \int \int_{\R^n} \mu \left[ \partial_t \rho + \dv{ pv } \right] \ud
  x \ud t = 
  -\int \int_{\R^n} \left[ \partial_t \mu + \nabla \mu \cdot
    v \right] \rho \ud x \ud t,
\end{equation}
where $\nabla$ denotes the spatial gradient operator. Notice that we
ignore the boundary terms from integration by parts as we will
eventually compute stationary conditions, and we are assuming fixed
initial conditions for $\rho_0$ and we assume that $\rho_{\infty}$
converges and thus cannot be perturbed when computing the variation of
the action integral. With this, we can formulate the action integral
with Lagrange multipliers as
\begin{align}
  \label{eq:action_lagrange}
  A &= \int \left[ T(v) - U(\phi) \right] \ud t 
      + \int \int_{\R^n} \lambda^T[ \partial_t \psi + (D\psi)v ] \ud
      x \ud t
      -\int \int_{\R^n} \left[ \partial_t \mu + \nabla \mu \cdot
      v \right] \rho \ud x \ud t,
\end{align}
where we have omitted the subscripts to avoid cluttering the
notation. Notice that the potential $U$ is now a function of $\psi$,
and the action depends now on $\rho, \psi, v$ and the Lagrange
multipliers $\mu, \lambda$.

We now compute variations of $A$ as we perturb the paths by variations
$\delta \rho$, $\delta v$ and $\delta \phi$ along the paths. The variation with
respect to $\rho$ is defined as $\delta A \cdot \delta \rho =
\left. \der{}{\varepsilon} A(\rho + \varepsilon \delta \rho, v, \psi)
\right|_{\varepsilon =0 }$, and the other variations are defined in a
similar fashion. By computing these variations, we get the following
stationary equations:

\begin{theorem} \label{thrm:stationary_lagrange_mult}
  The stationary conditions of the path for the action \eqref{eq:action_lagrange}
  are
  \begin{align}
    \partial_t \lambda + (D\lambda )v + \lambda \dv{v}  &= (\nabla\psi)^{-1}\nabla
                                                          U(\phi)\\
    \rho v + (\nabla \psi) \lambda - \rho\nabla\mu &= 0 \\
    \partial_t \mu + \nabla\mu \cdot v &= \frac 1 2 |v|^2
  \end{align}
  where $\nabla U(\phi)\in T_{\phi}\diff$ denotes the functional
  gradient of $U$ with respect to $\phi$ (see
  Appendix~\ref{app:funct_grads}), and $\nabla \mu, \nabla \psi$ are
  spatial gradients. The original constraints
  \eqref{eq:transport_backward} on the mapping and the continuity
  equation \eqref{eq:continuity_eqn} are part of the stationary
  conditions.
\end{theorem}
\begin{proof}
  See Appendix~\ref{app:stat_cond_nondissip}.
\end{proof}

While the previous theorem does give the stationary conditions and
evolution of the Lagrange multipliers, in order to define a forward
evolution method where the initial conditions for the density, mapping
and velocity are given, we would need initial conditions for the
Lagrange multipliers, which are not known from the calculation leading
to Theorem~\ref{thrm:stationary_lagrange_mult}. Therefore, we will now
eliminate the Lagrange multipliers and rewrite the evolution equations
in terms of forward equations for the velocity, mapping and
density. This leads to the following theorem:

\begin{theorem}[Evolution Equations for the Path of Least Action]
  \label{thrm:evol_final_non_dissip}
  The stationary conditions for the path of the action integral
  \eqref{eq:action} subject to the constraints \eqref{eq:phi_evol} on
  the mapping and the continuity equation \eqref{eq:continuity_eqn}
  are given by the forward evolution equation
  \begin{equation} \label{eq:evol_velocity}
    \partial_t v = -(Dv)v -\frac{1}{\rho} \nabla U(\phi),
  \end{equation}
  which describes the evolution of the velocity. The forward evolution
  equation for the diffeomorphism is given by \eqref{eq:phi_evol},
  that of its inverse mapping is given by
  \eqref{eq:transport_backward}, and the forward evolution of its
  density is given by \eqref{eq:continuity_eqn}.
\end{theorem}
\begin{proof}
  See Appendix~\ref{app:velocity_evol}.
\end{proof}

\begin{remark}[Relation to Euler's Equations]
  The left hand side of the equation (along with the continuity
  equation) is the left hand side of the \emph{compressible Euler
    Equation} \cite{marsden2013introduction}, which describes the
  motion of a perfect fluid (i.e., assuming no heat transfer or
  viscous effects). The difference is that the right hand side in
  \eqref{eq:evol_velocity} is the gradient of the potential, which we
  seek to optimize, that depends on the diffeomorphism that is the
  integral of the velocity over time, rather than the gradient of
  pressure that is purely a function of density in the Euler
  equations.
\end{remark}

With this theorem, it is now possible to numerically compute the
stationary path of the action, by starting with initial conditions on
the density, mapping and velocity. The velocity is
updated by \eqref{eq:evol_velocity}, the mapping is then updated by
\eqref{eq:phi_evol}, and the density is updated by
\eqref{eq:continuity_eqn}. Note that the density at each time impacts
the velocity as seen in \eqref{eq:evol_velocity}. These equations are
a set of coupled partial differential equations. They describe the
path of stationary action when the action integral does not arise from a
system that has dissipative forces. Notice  the
velocity evolution is a natural analogue of Newton's
equations. Indeed, if we consider the material derivative, which
describes the time rate of change of a quantity subjected to a time
dependent velocity field, then one can write the velocity evolution
\eqref{eq:evol_velocity} as follows.
\begin{theorem}[Equivalence of Critical Paths of Action to Newton's 2nd
  Law]
  The velocity evolution \eqref{eq:evol_velocity} derived as the
  critical path of the action integral \eqref{eq:action} is 
  \begin{equation} \label{eq:newton_law}
    \rho\frac{Dv}{Dt} = -\nabla U(\phi),
  \end{equation}
  where $\frac{Df}{Dt} := \partial_t f + (Df)v$ is the material
  derivative.
\end{theorem}
\begin{proof}
  This is consequence of the definition of material derivative.
\end{proof}
The material derivative is obtained by taking the time derivative of
$f$ along the path $t\to \phi(t,x)$, i.e., $\der{}{t}
f(t,\phi(t,x))$. Therefore, $Dv/Dt$ is the derivative of velocity
along the path. The equation \eqref{eq:newton_law} says the time rate
of change of velocity times density is equation to minus the gradient
of the potential, which is Newton's 2nd law, i.e., the mass times
acceleration is equal to the force, which is given by the gradient of
the potential in a conservative system.

The evolution described by the equations above will not converge. This
is because the total energy is conserved, and thus the system will
oscillate over a (local) minimum of the potential $U$, forever, unless
the initialization is at a stationary point of the potential $U$. In
practice, due to discretization of the equations, which require
entropy preserving schemes \cite{sethian1999level}, the implementation
will dissipate energy and the evolution equations eventually converge.

\subsubsection{Viscosity Solution and Regularity}
An important question is whether the evolution equations given by
Theorem~\ref{thrm:evol_final_non_dissip} maintain that the mapping
$\phi_t$ remains a diffeomorphism given that one starts the evolution
with a diffeomorphism. This is of course important since all of the
derivations above were done assuming that $\phi$ is a diffeomorphism,
moreover for many applications one wants to maintain a diffeomorphic
mapping. The answer is affirmative since to define a solution of
\eqref{eq:evol_velocity}, we define the solution as the
\emph{viscosity solution} (see e.g.,
\cite{crandall1983viscosity,rouy1992viscosity,sethian1999level}). The
viscosity solution is defined as the limit of the equation
\eqref{eq:evol_velocity} with a diffusive term of the velocity added
to the right hand side, as the diffusive coefficient goes to
zero. More precisely
\begin{equation} \label{eq:evol_velocity_visc}
  \partial_t v_{\varepsilon} = -(Dv_{\varepsilon})v_{\varepsilon} +
  \varepsilon \Delta v_{\varepsilon} -\frac{1}{\rho} \nabla U(\phi),
\end{equation}
where $\Delta$ denotes the spatial Laplacian, which is a smoothing
operator. This leads to a smooth ($C^{\infty}$) solution due to the
known smoothing properties of the Laplacian. The viscosity solution is
then $v = \lim_{\varepsilon\to 0} v$. In practice, we do not actually
add in the diffusive term, but rather approximate the effects with
small $\varepsilon$ by using entropy conditions in our numerical
implementation. One may of course add the diffusive term to induce
more regularity into the velocity and thus into the mapping
$\phi$. Since the velocity is smooth ($C^{\infty}$), the integral of a
smooth vector field will result in a diffeomorphism
\cite{ebin1970groups}.

\subsubsection{Discussion}
An important property of these evolution equations, when compared to
virtually all previous image registration and optical flow methods is
the lack of need to compute inverses of differential operators, which
are global smoothing operations, and are expensive. Typically, in
optical flow (such as the classical Horn \& Schunck
\cite{horn1981determining}) or LDDMM \cite{beg2005computing} where one
computes Sobolev gradients, one needs to compute inverses of
differential operators, which are expensive. Of course one could
perform standard gradient descent, which does not typically require
computing inverses of differential operators, but gradient descent is
known not to be feasible and it is hard to numerically implement
without significant pre-processing, and easily gets stuck in what are
effectively numerical local minima. The equations in
Theorem~\ref{thrm:evol_final_non_dissip} are all local, and
experiments suggest they are not susceptible to the problems that
plague gradient descent.

\subsubsection{Constant Density Case}
We now analyze the case when the density $\rho$ is chosen to be a
fixed constant, and we derive the evolution equations. In this case,
the kinetic energy simplifies as follows 
\begin{equation} \label{eq:kinetic_const_density}
  T(v) = \frac {\rho}{ 2 }\int_{ \phi(\R^n)  } |v(x)|^2 \ud x.
\end{equation}
We can define the action integral as before \eqref{eq:action} with the
previous definition of kinetic energy, and we can derive the
stationary conditions by defining the following action integral
incorporating the mapping constraint
\eqref{eq:transport_backward}. This gives the modified action integral
as
\begin{align}
  \label{eq:action_lagrange_const_density}
  A &= \int \left[ T(v) - U(\phi) \right] \ud t 
      + \int \int_{\R^n} \lambda^T[ \partial_t \psi + (D\psi)v ] \ud
      x \ud t.
\end{align}
Note that the continuity equation is no longer imposed as a constraint
as the density is treated as a fixed constant. This leads to the
following stationary conditions.
\begin{theorem} \label{thrm:stationary_lagrange_mult_const_density}
  The stationary conditions of the path for the action
  \eqref{eq:action_lagrange_const_density} are
  \begin{align}
    \partial_t \lambda + (D\lambda )v + \lambda \dv{v}  &= (\nabla\psi)^{-1}\nabla
                                                          U(\phi)\\
    \rho v + (\nabla \psi) \lambda &= 0
  \end{align}
  where $\nabla U(\phi)\in T_{\phi}\diff$ denotes the functional
  gradient of $U$ with respect to $\phi$, and $\nabla \psi$ are
  spatial gradients. The original constraint
  \eqref{eq:transport_backward} on the mapping is part of the stationary
  conditions.
\end{theorem}
\begin{proof}
  The computation is similar to the non-constant density case
  Appendix~\ref{app:stat_cond_nondissip}. Note that stationary
  condition with respect to the mapping remains the same as the
  density constraint in the non-constant density case does not depend
  on the mapping. The stationary condition with respect to the
  velocity avoids the variation with respect to the density constraint
  in the non-constant density case, and remains the same except for
  the last term.
\end{proof}

As before, we can solve for the velocity evolution directly. This
results in the following result.
\begin{theorem}[Evolution Equations for the Path of Least Action]
  \label{thrm:evol_final_non_dissip_const_density}
  The stationary conditions for the path of the action integral
  \eqref{eq:action} with kinetic energy
  \eqref{eq:kinetic_const_density} subject to the constraint
  \eqref{eq:phi_evol} on the mapping is given by the forward evolution
  equation
  \begin{equation} \label{eq:evol_velocity_const_density}
    \partial_t v = -(Dv)v -(\nabla v)v -v\dv{v} - \frac{1}{\rho} \nabla U(\phi)
  \end{equation}
  The forward evolution
  equation for the diffeomorphism is given by \eqref{eq:phi_evol},
  and that of its inverse mapping is given by
  \eqref{eq:transport_backward}.
\end{theorem}
\begin{proof}
  We can apply Lemma~\ref{lem:lambda_t_w_t} in
  Appendix~\ref{app:velocity_evol} with $w = -\frac{1}{\rho} v$.
\end{proof}
The former equation \eqref{eq:evol_velocity_const_density} (without
the potential term) is known as the Euler-Poincar\'{e} equation (EPDiff), the
geodesic equation for the diffeomorphism group under the $L^2$ metric
\cite{miller2006geodesic}. This shows that one relationship between
Euler's equation and EPDiff is that Euler's equation is derived by a
time-varying density in the kinetic energy, which is optimized over
the mass distribution along with the velocity whereas EPDiff assumes a
constant mass density in the kinetic energy. The non-constant density
model (arising in Euler's equation) has a natural interpretation in
terms of Newton's equations.

\subsection{Acceleration with Energy Dissipation}
We now present the case of deriving the stationary conditions for a
system on the manifold of diffeomorphisms in which total energy
dissipates. This is important so the system will converge to a local
minima, and not oscillate about a local minimum forever, as the
evolution equations from the previous section. To do this, we consider
time varying scalar functions $a, b : [0,\infty) \to \R^+$, and define
the action integral, again defined on paths of diffeomorphisms, as
follows:
\begin{equation} \label{eq:action_diss}
  A = \int \left[ a_t T(v_t) - b_t U(\phi_t) \right] \ud t,
\end{equation}
where $a_t, b_t$ denote the values of the scalar at time $t$. We may
again go through finding the stationary conditions subject to the
mapping constraint \eqref{eq:transport_backward} and the continuity
equation constraint \eqref{eq:continuity_eqn}, with Lagrange multiplier
and then derive the forward evolution equations. The final result is
as follows:

\begin{theorem}[Evolution Equations for the Path of Least Action]
  \label{thrm:evol_final_dissip}
  The stationary conditions for the path of the action integral
  \eqref{eq:action_diss} subject to the constraints \eqref{eq:phi_evol} on
  the mapping and the continuity equation \eqref{eq:continuity_eqn}
  are given by the forward evolution equation
  \begin{equation} \label{eq:evol_velocity_dissp}
    a \partial_t v + a(Dv)v + (\partial_ta)v= -\frac{b}{\rho} \nabla U(\phi),
  \end{equation}
  which describes the evolution of the velocity. The same evolution
  equations as Theorem~\ref{thrm:evol_final_non_dissip} for the
  mappings \eqref{eq:phi_evol} and \eqref{eq:transport_backward}, and
  density hold \eqref{eq:continuity_eqn}.
\end{theorem}
\begin{proof}
  See Appendix~\ref{app:stationary_dissip}.
\end{proof}

If we consider certain forms of $a$ and $b$, then one can arrive at
various generalizations of Nesterov's schemes. In particular, the
choice of $a$ and $b$ below are those considered in
\cite{wibisono2016variational} to explain various versions of
Nesterov's schemes, which are optimization schemes in finite
dimensions.
\begin{theorem}[Evolution Equations for the Path of Least Action:
  Generalization of Nesterov's Method]
  If we choose  
  \[
    a_t = e^{\gamma_t-\alpha_t} \quad  \mbox{and} \quad b_t = e^{\alpha_t+\beta_t+\gamma_t} 
  \]
  where
  \[
    \alpha_t = \log p - \log t, \quad  
    \beta_t = p\log t + \log C, \quad  \gamma_t = p\log t,
  \]
  $C>0$ is a constant, and $p$ is a positive integer, then we will
  arrive at the evolution equation
  \begin{equation}
    \partial_t v = -\frac {p+1}{t} v - (Dv)v - \frac{1}{\rho}Cp^2
    t^{p-2}\nabla U(\phi).
  \end{equation}
  In the case $p=2$ and $C=1/4$ the evolution reduces to
  \begin{equation}  \label{eq:vel_evol_dissp_nesterov}
    \partial_t v = -\frac {3}{t} v - (Dv)v - \frac{1}{\rho}
    \nabla U(\phi).
  \end{equation}
\end{theorem}
The case $p=2$ was considered in \cite{wibisono2016variational} as the
continuum equivalent to Nesterov's original scheme in finite
dimensions. We can notice that this evolution equation is the same as
the evolution equations for the non-dissipative case
\eqref{eq:evol_velocity}, except for the term $-(3/t) v$. One can
interpret the latter term as a frictional dissipative term, analogous
to viscous resistance in fluids. Thus, even in this case the equation
has a natural interpretation that arises from Newton's laws.

\subsection{Second Order PDE for Acceleration}

We now convert the system of PDE for the forward mapping and velocity
into a second order PDE in the forward mapping itself. Interestingly,
this eliminates the non-linearity from the non-potential terms.

\begin{theorem}[Second Order PDE for the Forward Mapping]
  The accelerated optimization, arising from the stationarity of the
  action integral \eqref{eq:action_diss}, given by the system of PDE
  defined by \eqref{eq:evol_velocity_dissp} and the forward mapping
  \eqref{eq:phi_evol} is 
  \begin{equation}
    a\pder{^2 \phi}{t^2} + (\partial_t a)\pder{\phi}{t} +
    \frac{b}{\rho_0} \widetilde{\nabla} U(\phi) = 0,
  \end{equation}
  where $\rho_0$ is the initial density,
  $\widetilde{\nabla} U(\phi) = [\nabla
  U(\phi)\circ\phi]\det{\nabla\phi}$ is the gradient defined on the
  un-warped domain, i.e.,
  $\delta A \cdot \delta \phi = \int_{\R^n} \widetilde{\nabla}
  U(\phi)(x) \cdot \delta \phi(x) \ud x$ is satisfied for all
  perturbations $\delta \phi$ of $\phi$.
\end{theorem}

\begin{proof}
  We differentiate the definition of the forward mapping in time to
  obtain \eqref{eq:phi_evol} and substituting the velocity evolution
  \eqref{eq:evol_velocity_dissp}:
  \begin{align*}
    \partial_{tt} \phi &= (\partial_t v)\circ \phi + [
                         (Dv)\circ\phi] \partial_t \phi \\
    &= -[ (Dv)\circ\phi ] v\circ\phi - \frac{\partial_t a}{a}
      v\circ\phi - \frac{b}{a} \frac{1}{\rho\circ\phi} \nabla
      U(\phi)\circ\phi + [ (Dv)\circ\phi ]\partial_t \phi \\
    &= - \frac{\partial_t a}{a} \partial_t \phi - \frac{b}{a} \frac{1}{\rho\circ\phi} \nabla
      U(\phi)\circ\phi
  \end{align*}
  We note the following for any $B\subset \R^n$, because of mass
  preservation, we have that
  \[
    \int_{B} \rho_0(x) \ud x = \int_{\phi(B)} \rho_t(y) \ud y = 
    \int_{B} \rho_t(\phi(x)) \det{\nabla\phi(x)} \ud x,
  \]
  where the last equality is obtained by a change of variables. Since
  we can take $B$ arbitrarily small, $\rho_0(x) = \rho_t(\phi(x))
  \det{\nabla\phi(x)}$. Using this last formula, we see that 
  \[
    \frac{1}{\rho\circ\phi} \nabla U(\phi)\circ\phi = 
    \frac{1}{\rho_0} \widetilde{\nabla} U(\phi),
  \]
  which proves the proposition.
\end{proof}

\subsection{Illustrative Potential Energy for Diffeomorphisms}
We now consider a standard potential for illustrative purposes in
simulations, and derive the gradient. The objective is for the
evolution equations in the previous section to minimize the potential,
which is a function of the mapping. Our evolution equations in the
previous section are general and work with \emph{any} potential; our
purpose in this section is not to advocate a particular potential, but
to show how the gradient of the potential is computed so that it can
be used in the evolution equations in the previous section. We
consider the standard Horn \& Schunck model for optical flow defined as 
\begin{equation} \label{eq:potential_HS}
  U(\phi) = \frac 1 2 \int_{\R^n} |I_1(\phi(x)) - I_0(x)|^2 \ud x +
  \frac 1 2 \alpha
  \int_{\R^n} |\nabla (\phi(x) - x )|^2 \ud x,
\end{equation}
where $\alpha>0$ is a weight, and $I_0, I_1$ are images. The first term is
the data fidelity which measures how close $\phi$ deforms $I_1$ back
to $I_0$ through the squared norm, and the second term penalizes
non-smoothness of the displacement field, given by $\phi(x)-x$ at the
point $x$. Notice that the potential is a function of only the mapping
$\phi$, and not the velocity. 

We now compute the functional gradient of $U$ with respect to the
mapping $\phi$, denoted by the expression $\nabla U(\phi)$. This
gradient is defined by the relation (see
Appendix~\ref{app:funct_grads})
$\delta U \cdot \delta \phi = \int_{\phi(\R^n)} \nabla U(\phi) \cdot
\delta \phi \ud x $, i.e., the functional gradient satisfies the
relation that the $\mathbb{L}^2$ inner product of it with any
perturbation $\delta \phi$ of $\phi$ is equal to the variation of the
potential $U$ with respect to the perturbation $\delta \phi$. With
this definition, one can show that (see
Appendix~\ref{app:funct_grads})
\begin{equation} \label{eq:potential_illustrative}
  \nabla U(\phi) = 
  \left[ ( I_1 - I_0\circ\psi ) \nabla I_1 - \alpha (\Delta \phi)\circ\psi \right]
  \det \nabla \psi,
\end{equation}
where $\det$ denotes the determinant.

We can also see that the gradient defined on the un-warped domain is 
\begin{equation}
  \widetilde{\nabla} U(\phi) =
  ( I_1\circ\phi - I_0 ) \nabla I_1\circ\phi - \alpha \Delta \phi,
\end{equation}
therefore, the generalization of Nesterov's method on the original
domain itself, in this case is 
\begin{equation}
    \pder{^2 \phi}{t^2} + \frac 3 t  \pder{\phi}{t} -
    \frac{\alpha}{\rho_0} \Delta \phi +
    \frac{1}{\rho_0}   ( I_1\circ\phi - I_0 ) \nabla I_1\circ\phi  = 0,
\end{equation}
which is a damped \emph{wave equation}.

\section{Experiments}
We now show some examples to illustrate the behavior of our
generalization of accelerated optimization to the infinite dimensional
manifold of diffeomorphisms. We compare to standard (Riemannian $L^2$)
gradient descent to illustrate how much one can gain by incorporating
acceleration, which requires little additional effort over gradient
descent. Over gradient descent, acceleration requires only to update
the velocity by the velocity evolution in the previous section, and
the density evolution. Both these evolutions are cheap to compute
since they only involve local updates. Note the gradient descent of
the potential $U$ is given by choosing $v = -\nabla U(\phi)$, the
other evolution equation for the mapping $\phi$ \eqref{eq:phi_evol}
and $\psi$ \eqref{eq:transport_backward} remains the same, and no
density evolution is considered. We note that we implement the
equations as they are, and there is no additional processing that is
now common in optical flow methods (e.g., no smoothing images nor
derivatives, no special derivative filters, no multi-scale techniques,
no use of robust norms, median filters, etc). Although our equations
are for diffeomorphisms on all of $\R^n$, in practice be have finite
images, and the issue of boundary conditions come up. For simplicity
to illustrate our ideas, we choose periodic boundary conditions. We
should note that our numerical scheme (see
Appendix~\ref{app:discretization}) for implementing accelerated
gradient descent is quite basic and not final, and a number of speed
ups and / or refinements to the numerics can be done, which we plan to
explore in the near future. Thus, at this point we do not compare the
method to current optical flow techniques since our numerics are not
finalized. Our intention is to show the promise of acceleration and
that simply by using acceleration, one can get an impractical
algorithm (gradient descent) to become practical, especially with
respect to speed.

In all the experiments, we choose the step size to satisfy CFL
conditions. For ordinary gradient descent we choose
$\Delta t < 1/(4\alpha)$, for accelerated gradient descent we have the
additional evolution of the velocity
\eqref{eq:vel_evol_dissp_nesterov}, and our numerical scheme has CFL
condition $\Delta t < 1/\max_{x\in\Omega} \{ |v(x)|, |Dv(x)|
\}$. Also, because there is a diffusion according to regularity, we
found that
$\Delta t < 1/(4\alpha \cdot \max_{x\in\Omega} \{ |v(x)|, |Dv(x)| \})$
gives stable results, although we have not done a proper Von-Neumann
analysis, and in practice we do see we can choose a higher step
size. The step size for accelerated gradient descent is lower in our
experiments than accelerated gradient descent.

In all experiments, the initialization is $\phi(x) = \psi(x) = x$,
$v(x)=0$, and $\rho(x) = 1/|\Omega|$ where $|\Omega|$ is the area of
the domain of the image.

\subsubsection{Convergence analysis}
In this experiment, the images are two white squares against a black
background. The sizes of the squares are $50\times 50$ pixels wide,
and the square (of size $20\times 20$) in the first image is
translated by $10$ pixels to form the second image. Small images are
chosen due to the fact gradient descent is too impractically slow for
reasonable sized images without multi-scale approaches that even
modest sized images (e.g., $256\times 256$) do not converge in a
reasonable amount of time, and we will demonstrate this in an
experiment later. Figure~\ref{fig:potential_vs_iter} shows the plot of
the potential energy \eqref{eq:potential_HS} of both gradient descent
and accelerated gradient descent as the evolution progresses. Here
$\alpha = 5$ (images are scaled between 0 and 1). Notice that
accelerated gradient descent very quickly accelerates to a global
minimum, surpasses the global minimum and then oscillates until the
friction term slows it down and then it converges very quickly. Notice
that this behavior is expected since accelerated gradient descent is
not a strict descent method (it does not necessarily decrease the
potential energy each step).  Gradient descent very slowly decreases
the energy each iteration and eventually converges.

\begin{figure}
  \centering
  \includegraphics[totalheight=2.5in]{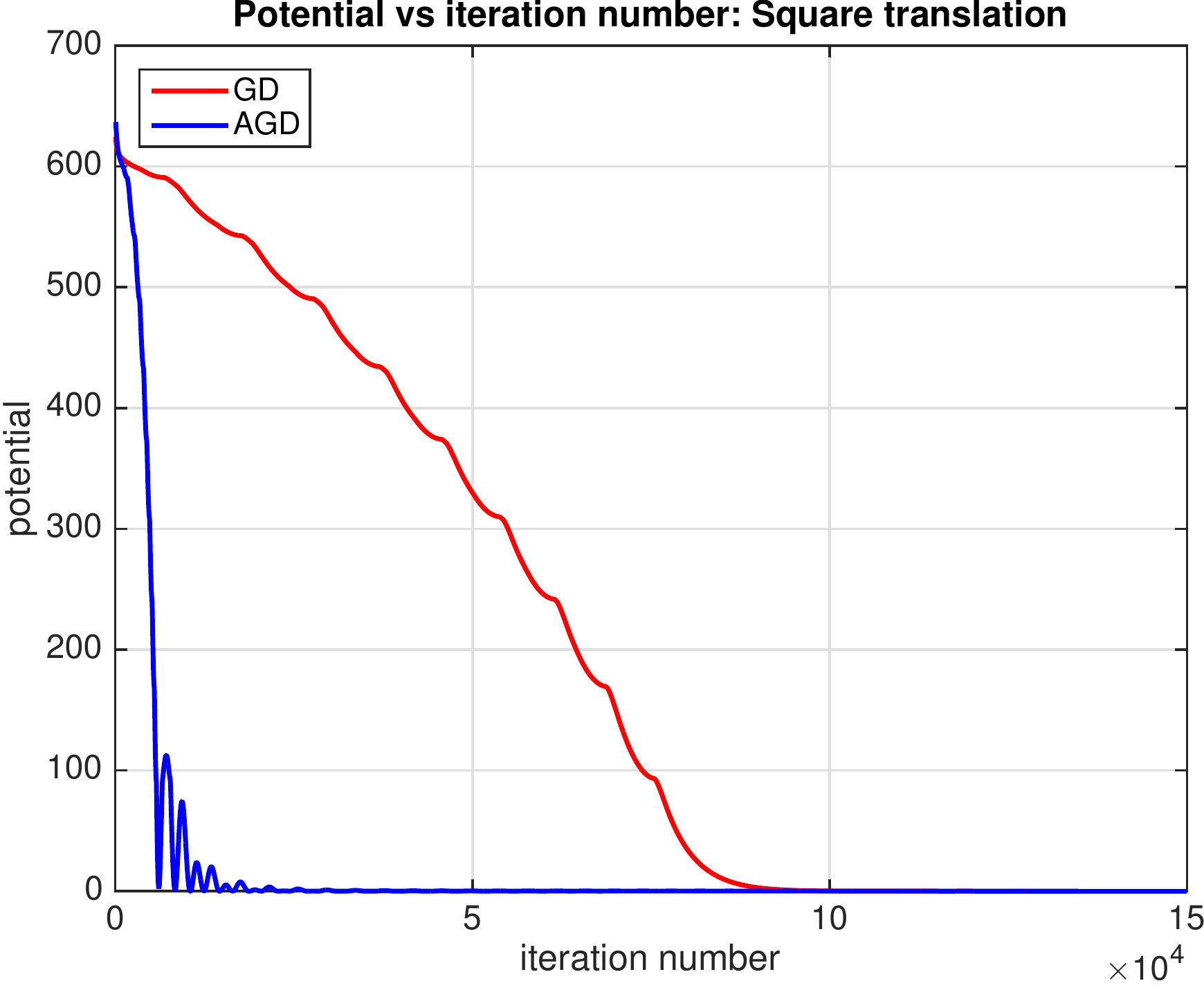}
  \caption{{\bf Convergence Comparison}: Two binary images with
    squares in which the square is translated are registered. The
    value of the functional (to be minimized) versus the iteration
    number is shown for both gradient descent (GD) and accelerated
    gradient descent (AGD).}
  \label{fig:potential_vs_iter}
\end{figure}

We now repeat the same experiment, but with different images to show
that this behavior is not restricted to the particular choice of
images, one a translation of the other. To this end, we choose the
images again to be $50\times 50$. The first image has a square that is
$17\times 17$ and the second image has a rectangle of size
$20\times 14$ and is translated by $8$ pixels. We choose the
regularity $\alpha=2$, since the regularity should be chosen smaller
to account for the stretching and squeezing, resulting in a non-smooth
flow field. A plot of results of this simulation is shown in
Figure~\ref{fig:potential_vs_iter_scaling}. Again accelerated gradient
accelerates very quickly at the start, then oscillates and the
oscillations die down and then it converges. This time the potential
does not go to zero since the final flow is not a translation and thus
the regularity term is non-zero. Gradient descent converges faster
than the case of translation due to larger $\alpha$ and thus larger
step size. However, it appears to be stuck in a higher energy
configuration. In fact, gradient descent has not fully converged -
gradient descent is slow in adapting to the scale changes and becomes
extremely slow in stretching and squeezing in different directions. We
verify that gradient descent has not fully converged by plotting just
the first term of the potential, i.e., the reconstruction error, which
is zero for accelerated gradient descent at convergence, indicating
that the flow correctly reconstructs $I_0$ from $I_1$. On the other
hand, gradient descent has an error of about $50$, indicating the flow
does not fully warp $I_1$ to $I_0$, and therefore it not the correct
flow. This does not appear to be a local minimum, just slow convergence.

\begin{figure}
  \centering
  \includegraphics[totalheight=2.5in]{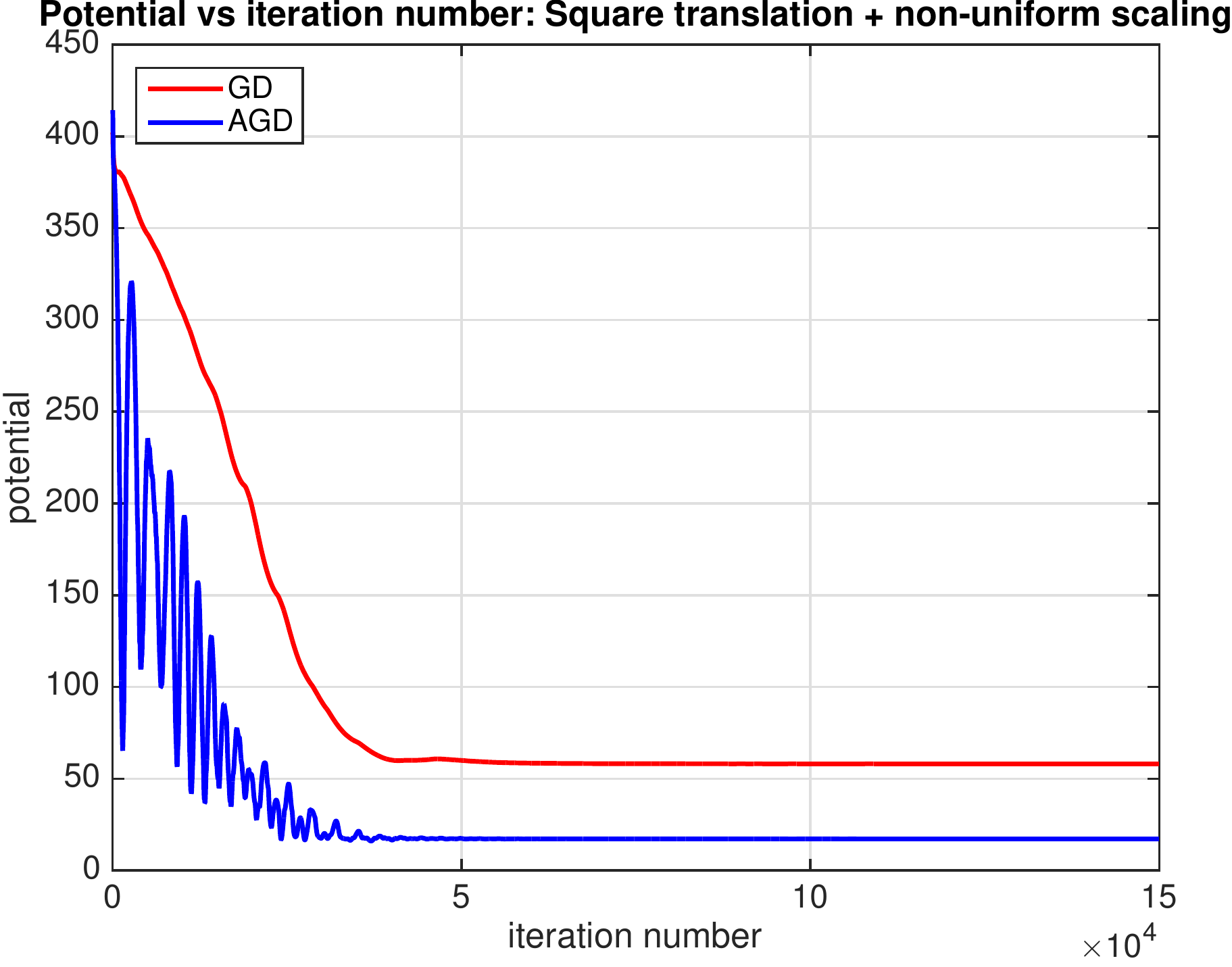}
  \includegraphics[totalheight=2.5in]{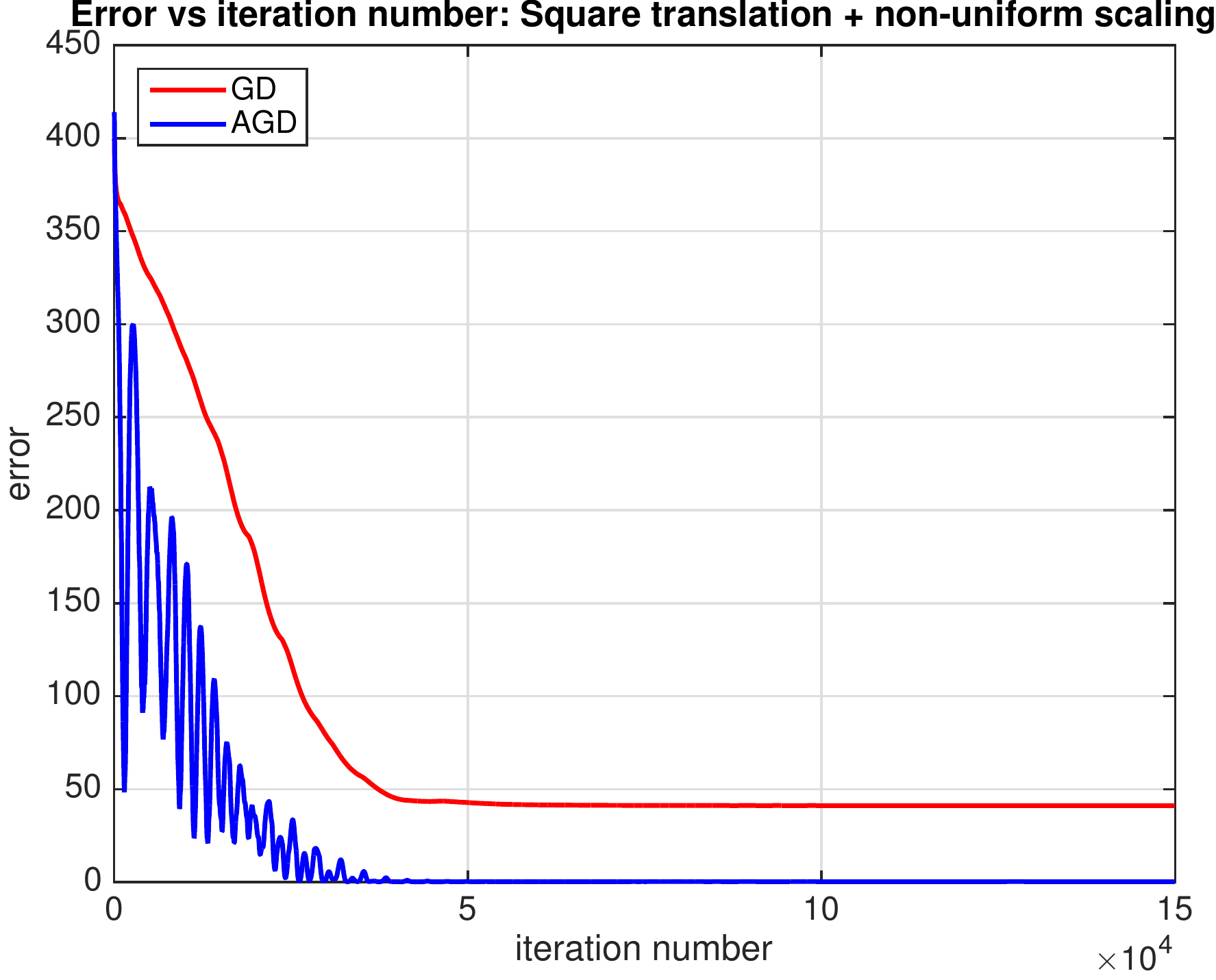}
  \caption{{\bf Convergence Comparison}: Two images are registered,
    each are binary images. The first is a square and the second image
    is a translated and non-uniformly scaled version of the square in
    the first image. [Left]: The cost functional to be minimized
    versus the iteration number is shown for both gradient descent
    (GD) and accelerated gradient descent (AGD). AGD converges to a
    lower energy solution quicker. [Right]: Note that GD did not fully
    converge as the convergence is extremely slow in obtaining fine
    scale details of the non-uniform scaling. This is verified by
    plotting the image reconstruction error:
    $\|I_1\circ \phi - I_0\|$, which
    shows that AGD reconstructs $I_0$ with zero error. }
  \label{fig:potential_vs_iter_scaling}
\end{figure}

We again repeat the same experiment, but with real images from a
cardiac MRI sequence, in which the heart beats. The transformation
relating the images is a general diffeomorphism that is not easily
described as in the previous experiments. The images are of size
$256\times 256$. We choose $\alpha=0.02$. A plot of the potential
versus iteration number for both gradient descent (GD) and accelerated
gradient descent (AGD) is shown in the left of
Figure~\ref{fig:cardiac_expt}.  The convergence is quicker for
accelerated gradient descent. The right of
Figure~\ref{fig:cardiac_expt} shows the original images and the images
warped under both the result from gradient descent and accelerated
gradient descent, and that they both produce a similar correct warp,
but accelerated gradient obtains the warp in much fewer iterations.

\begin{figure}
  \centering
  \begin{minipage}{0.3\textwidth}
    \centering
    \includegraphics[totalheight=2.5in]{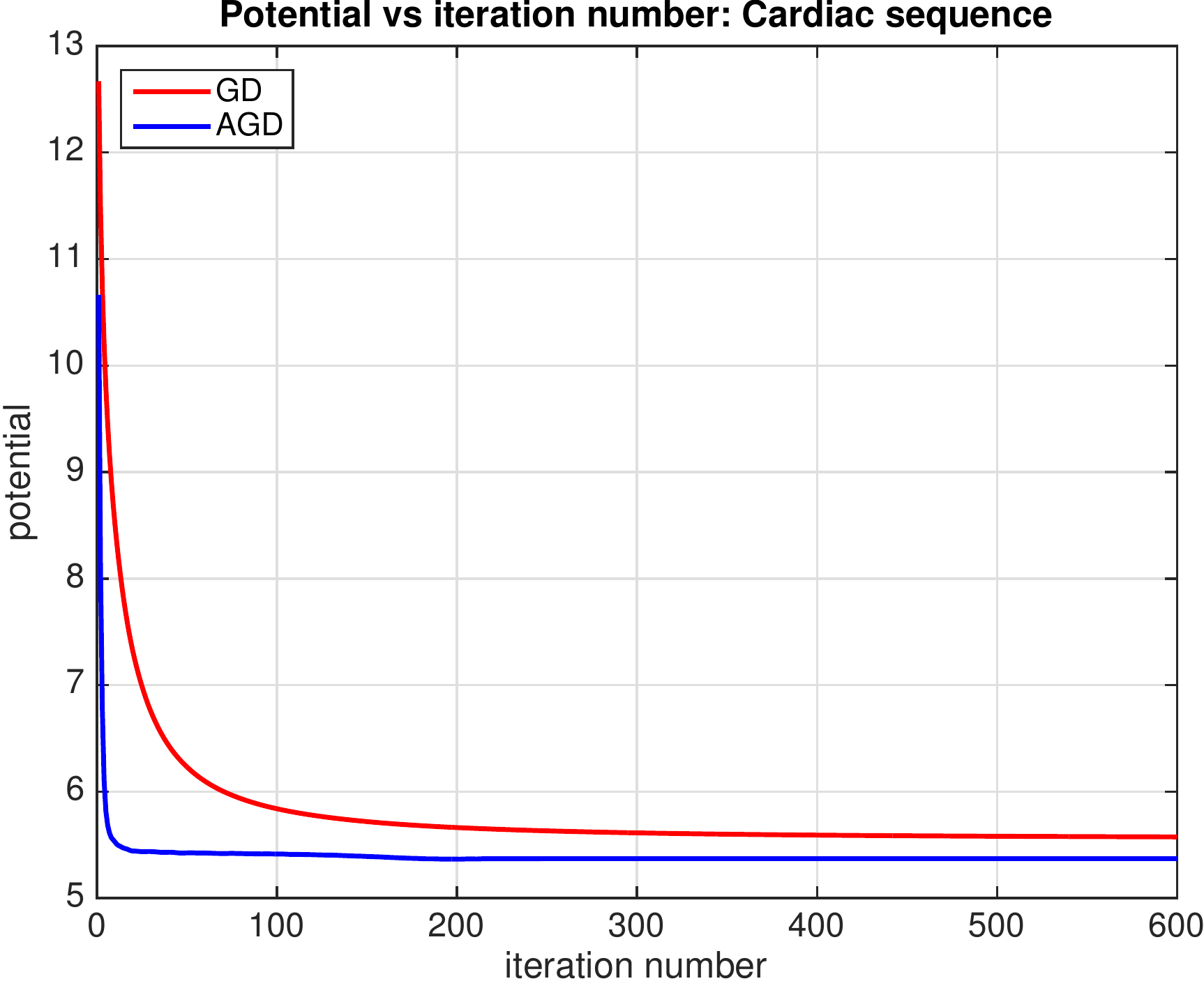}
  \end{minipage}%
  \begin{minipage}{0.8\textwidth}
    \centering
    \begin{tabular}{c@{\hspace{0.05in}}c}
      $I_1$ & $I_0$ \\
      \includegraphics[totalheight=1.5in]{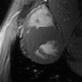} &
      \includegraphics[totalheight=1.5in]{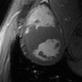} \\
      $I_1\circ \phi_{gd}$ & $I_1\circ \phi_{agd}$ \\
      \includegraphics[totalheight=1.5in]{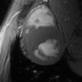} &
      \includegraphics[totalheight=1.5in]{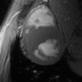}
    \end{tabular}
  \end{minipage}
  \caption{{\bf Convergence Comparison}: Two MR cardiac images from a
    sequence are registered. The images are related through a general
    deformation. [Left]: A plot of the potential versus the iteration
    number in the minimization using gradient descent (GD) and
    accelerated gradient descent (AGD). AGD converges at a quicker
    rate. [Right]: The original images and the back-warped images using
    the recovered diffeomorphisms. Note that $I_1\circ\phi$ should
    appear close to $I_0$. Both methods seem to recover a similar
    transformation, but AGD recovers it faster.}
  \label{fig:cardiac_expt}
\end{figure}

\subsubsection{Convergence analysis versus parameter settings}
We now analyze the convergence of accelerated gradient descent and
gradient descent as a function of the regularity $\alpha$ and the
image size. To this end, we first analyze an image pair of size
$50\times 50$ in which one image has a square of size $16\times 16$
and the other image is the same square translated by $7$ pixels. We
now vary $\alpha$ and analyze the convergence. In the left plot of
Figure~\ref{fig:speed_vs_alpha}, we show the number of iterations
until convergence versus the regularity $\alpha$. As $\alpha$
increases, the number of iterations for both gradient descent and
accelerated gradient descent increase as expected since there is a
inverse relationship between $\alpha$ and the step size. However, the
number of iterations for accelerated gradient descent grows more
slowly. In all cases, the algorithm is run until the flow field
between successive iterations does not change according to a fixed
tolerance. In all cases, the flow achieves the ground truth flow.

Next, we analyze the number of convergence iterations versus the image
size. To this end, we again consider binary images with squares of
size $16\times 16$ and translated by $7$ pixels. However, we vary the
image size from $50\times 50$ to $200 \times 200$. We fix
$\alpha = 8$. Now we show the number of iterations to convergence
versus the image size. This is shown in the right plot of
Figure~\ref{fig:speed_vs_alpha}. Gradient descent is impractically
slow for all the sizes considered, and the number of iterations
quickly increases with image size (it appears to be an exponential
growth). Accelerated gradient descent, surprisingly, appears to have
very little or no growth with respect to the image size. Of course one
could use multi-scaling pyramid approaches to improve gradient
descent, but as soon as one goes to finer scales, gradient descent is
incredibly slow even when the images are related by small
displacements. Simple acceleration makes standard gradient descent
scalable with just a few extra local updates.

\begin{figure}
  \centering
  \includegraphics[totalheight=2.5in]{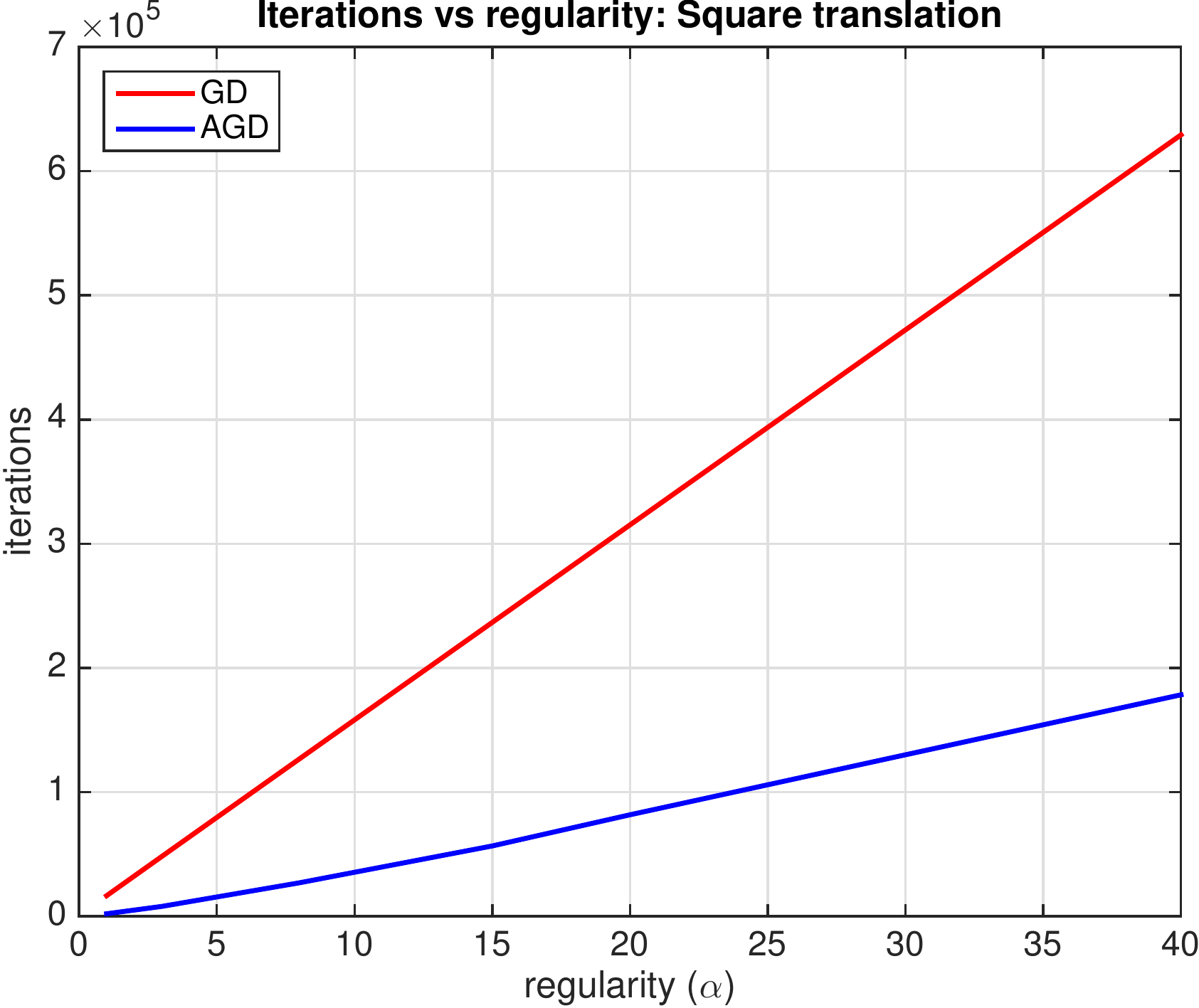}
  \includegraphics[totalheight=2.5in]{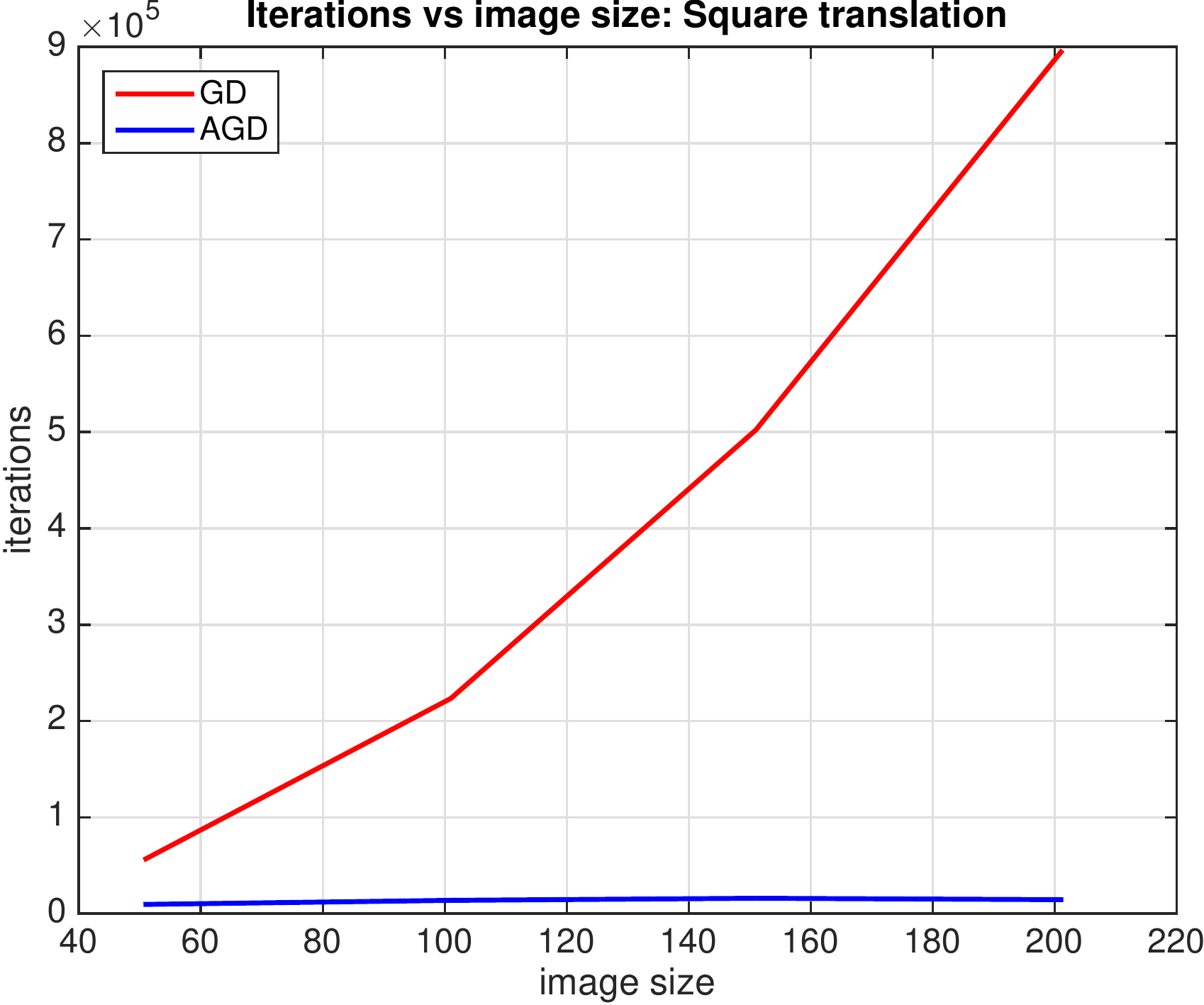}
  \caption{[Left]: {\bf Convergence Comparison as a Function of
      Regularity}: Two binary images (a square and a translated
    square) are registered with varying amounts of regularization
    $\alpha$ for gradient descent (GD) and accelerated gradient
    descent (AGD). [Right]: {\bf Convergence Comparison as a Function
      of Image Size}: We keep the squares in the images and $\alpha=3$
    fixed, but we vary the size (height and width) of the image and
    compare GD with AGD. Very quickly, gradient descent becomes
    impractical due to extremely slow convergence.}
  \label{fig:speed_vs_alpha} 
\end{figure}

\comment{
We repeat the previous experiment with the first image being a square
and the second image being a translated and scaled version of the
square to form a rectangle. The square is $40\times 40$ pixels and the
rectangle is $20 \times 60$ pixels. Again gradient descent
fails. Accelerated dissipative optimization (shown in
Figure~\ref{fig:square_trans_scaled}) correctly recovers the mapping.

\begin{figure}
  \label{fig:square_trans_scaled}
\end{figure}
}

\subsubsection{Analysis of Robustness to Noise}
We now analyze the robustness of gradient descent and accelerated
gradient descent to noise. We do this to simulate robustness to
undesirable local minima. We choose to use salt and pepper noise to
model possible clutter in the image. We consider images of size
$50\times 50$. We fix $\alpha = 1$ in all the simulations and vary the
noise level; of course one could increase $\alpha$ to increase
robustness to noise. However, we are interested in understanding the
robustness to noise of the optimization algorithms themselves rather
than changing the potential energy to better cope with noise.  First,
we consider a square of size $16\times 16$ in the first binary image
and the same square translated by $4$ pixels in the second image. We
plot the error in the flow (measured as the average endpoint error of
the flow returned by the algorithm against ground truth flow) versus
the noise level. The result is shown in the left plot of
Figure~\ref{fig:noise_stability}. This shows that accelerated gradient
descent degrades much slower than gradient
descent. Figure~\ref{fig:noise_stab_trans_img} shows visual comparison
of the final results where we show $I_1\circ\phi$ and compare it to
$I_0$ for both accelerated gradient descent and gradient descent.

We repeat the same experiment to show that this trend is not just with
this configuration of images. To this end, we experiment with
$50\times 50$ images one with a square of size $15\times 15$ and a
rectangle that is size $20\times 10$ and translated by $5$ pixels. We
again fix the regularity to $\alpha=1$. The result of the experiment
is plotted in the right of Figure~\ref{fig:noise_stability}. A similar
trend of the previous experiment is observed: accelerated gradient
descent degrades much less than gradient descent. Note we have
measured accuracy as the average reconstruction error with the
original (non-noisy) images. This is because the ground truth flow is
not known. Figure~\ref{fig:noise_stab_dil_img} shows visual comparison
of the final results.

\begin{figure}
  \centering
  \includegraphics[totalheight=2.5in]{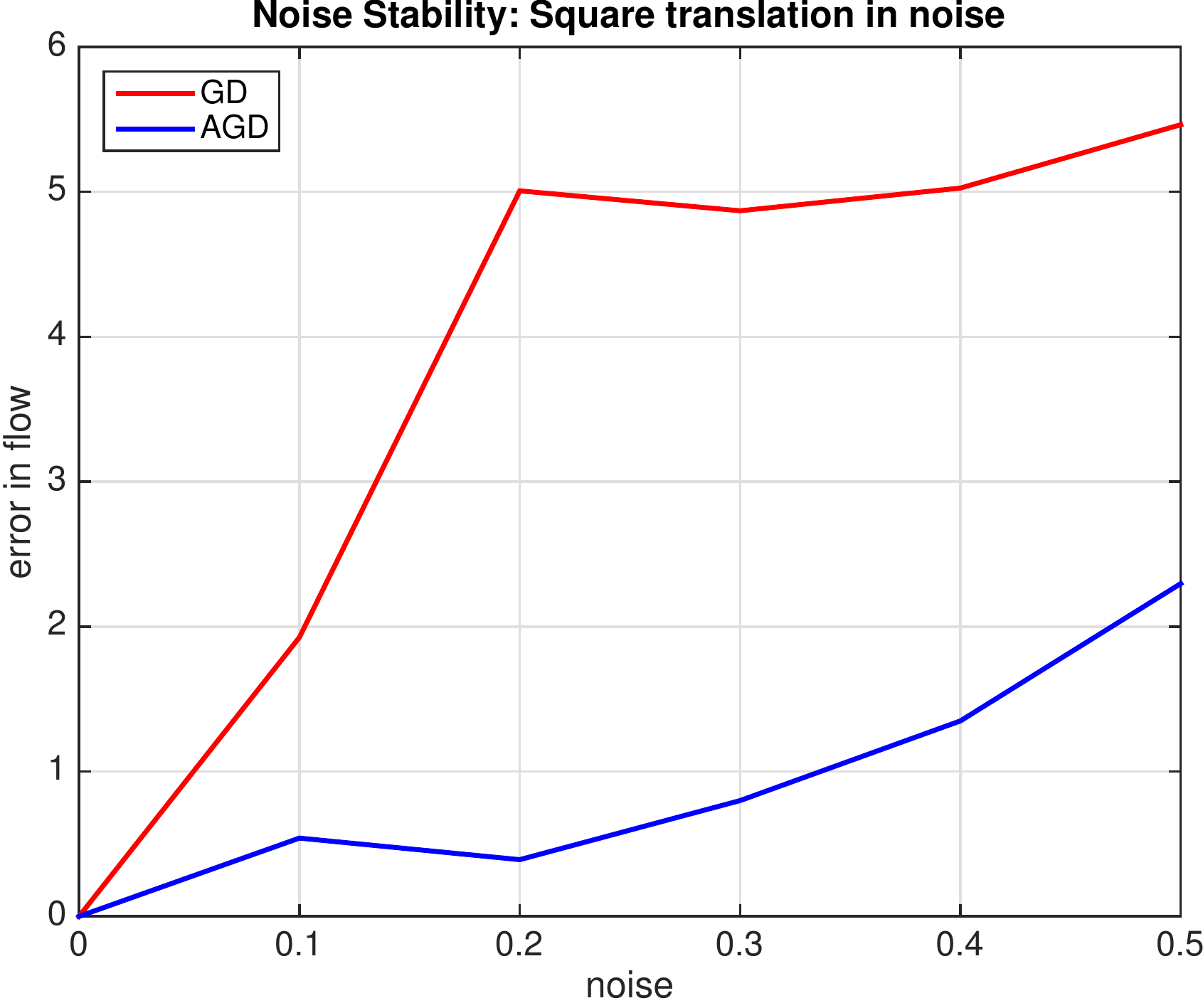}
  \includegraphics[totalheight=2.5in]{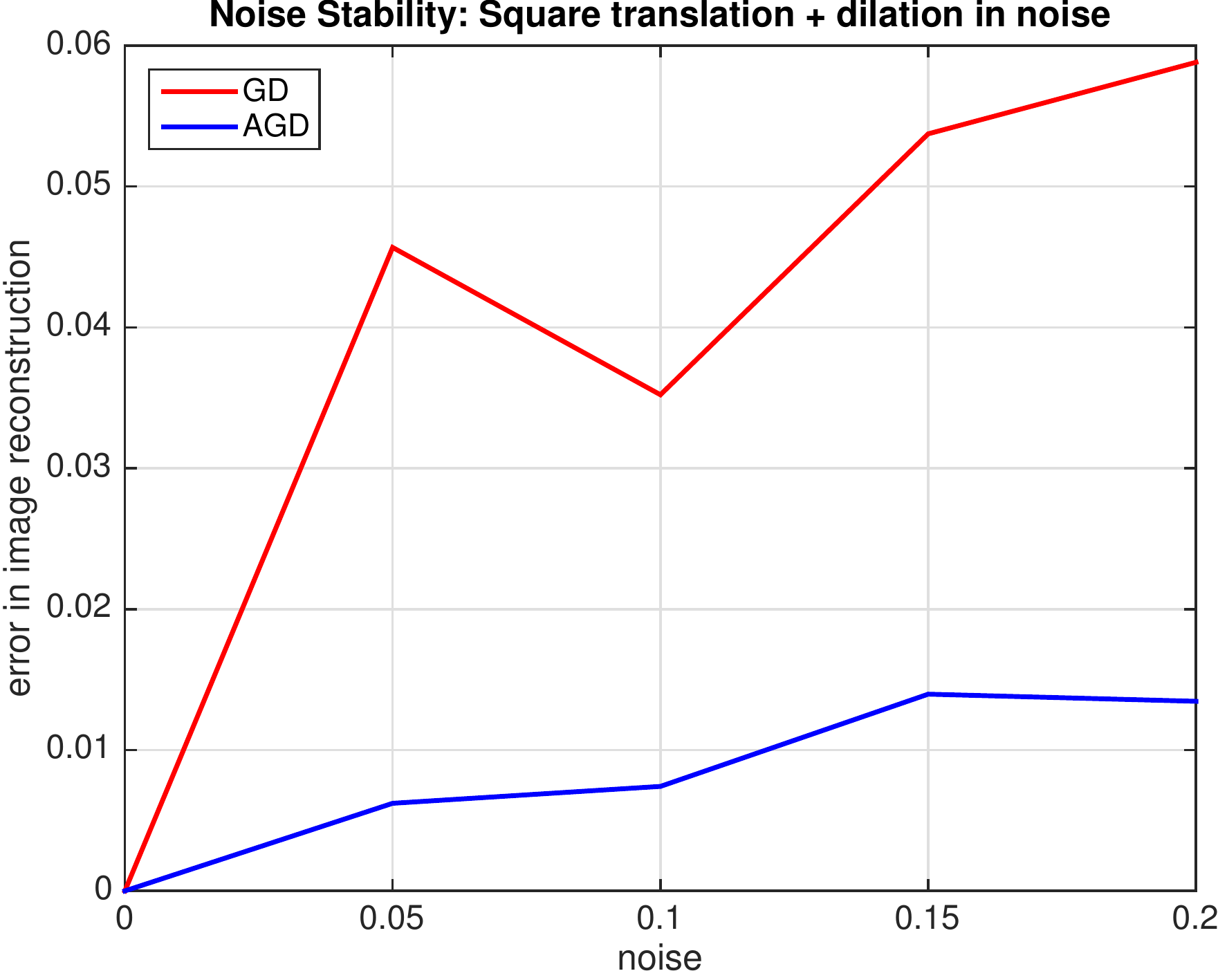}
  \caption{{\bf Analysis of Stability to Noise}: We add salt and
    pepper noise with varying intensity to binary images and then
    register the images. We plot the error in the recovered flow of
    both gradient descent (GD) and accelerated gradient descent (AGD)
    versus the level of noise. The value of $\alpha$ is kept
    fixed. The error is measured by the average endpoint error of the
    flow. [Left]: The first image is formed from a square and the
    second image is the same square but translated. [Right]: The first
    image is a square and the second image is the non-uniformly scaled
    and translated square. The error is measured as the average image
    reconstruction error.}
  \label{fig:noise_stability}
\end{figure}

\def\hExpt1img{0.75in}
\def\himgsp{0.5in}
\begin{figure}
  \centering
  {\small
  \begin{tabular}{ll}
    $\quad\quad\quad I_1\quad\quad$ $\quad \quad\quad I_0 \quad$ $\quad\quad\quad  I_1\circ\phi_{gd}$
    $\quad\quad I_1\circ\phi_{agd}$ & $\quad\quad\quad I_1\quad\quad$ $\quad \quad\quad I_0 \quad$ $\quad\quad\quad  I_1\circ\phi_{gd}$
    $\quad\quad I_1\circ\phi_{agd}$ \\
   \rotatebox{90}{$\quad\sigma = 0.0$}\includegraphics[totalheight=\hExpt1img]{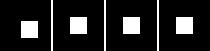} &
  \rotatebox{90}{$\quad\sigma = 0.1$}\includegraphics[totalheight=\hExpt1img]{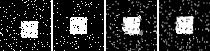}\\
  \rotatebox{90}{$\quad\sigma = 0.2$}\includegraphics[totalheight=\hExpt1img]{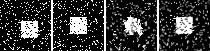} &
  \rotatebox{90}{$\quad\sigma = 0.3$}\includegraphics[totalheight=\hExpt1img]{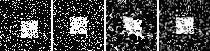}\\
  \rotatebox{90}{$\quad\sigma = 0.4$}\includegraphics[totalheight=\hExpt1img]{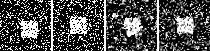} &
  \rotatebox{90}{$\quad\sigma = 0.5$}\includegraphics[totalheight=\hExpt1img]{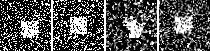}\\
  \end{tabular}
}
\caption{{\bf Visual Comparison on Square Translation in Noise
    Experiment}. The above show the visual results of the noise
  robustness experiment. For each row group of images: the two original images, the
  warped image by gradient descent, and the warped image by
  accelerated gradient descent. The last two images should resemble
  the second if the registration is correct.}
\label{fig:noise_stab_trans_img}
\end{figure}

\begin{figure}
  \centering
  {\small
  \begin{tabular}{ll}
    $\quad\quad\quad I_1\quad\quad$ $\quad \quad\quad I_0 \quad$ $\quad\quad\quad  I_1\circ\phi_{gd}$
    $\quad\quad I_1\circ\phi_{agd}$ & $\quad\quad\quad I_1\quad\quad$ $\quad \quad\quad I_0 \quad$ $\quad\quad\quad  I_1\circ\phi_{gd}$
    $\quad\quad I_1\circ\phi_{agd}$ \\
   \rotatebox{90}{$\quad\sigma = 0.00$}\includegraphics[totalheight=\hExpt1img]{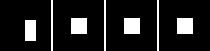} &
  \rotatebox{90}{$\quad\sigma = 0.05$}\includegraphics[totalheight=\hExpt1img]{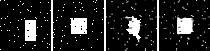}\\
  \rotatebox{90}{$\quad\sigma = 0.10$}\includegraphics[totalheight=\hExpt1img]{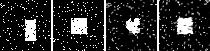} &
  \rotatebox{90}{$\quad\sigma = 0.15$}\includegraphics[totalheight=\hExpt1img]{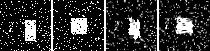}\\
  \rotatebox{90}{$\quad\sigma = 0.20$}\includegraphics[totalheight=\hExpt1img]{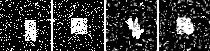} &
  \\
  \end{tabular}
}
\caption{{\bf Visual Comparison on Square Non-Uniform Scaling and
    Translation in Noise Experiment}. The above show the visual
  results of the noise robustness experiment. For each row group of
  images: the two original images, the warped image by gradient
  descent, and the warped image by accelerated gradient descent. The
  last two images should resemble the second if the registration is
  correct.}
\label{fig:noise_stab_dil_img}
\end{figure}

\comment{
We repeat the first experiment with the same images, but with the
images corrupted by salt and pepper noise to model clutter. 
The
results of non-dissipative accelerated gradient descent are shown in
Figure~\ref{fig:noisy_square}, which shows convergence to the desired
solution.

\begin{figure}
  \label{fig:noisy_square}
\end{figure}
}

\comment{
\subsubsection{Cardiac Image Registration}
Finally, we show accelerated non-dissipative optimization working on a
real image sequence pair. The pair of images is sampled from a cardiac
magnetic resonance image (MRI) sequence of a heart beating. The
mapping between frames is more than scaling and translating, and is
described by a more general diffeomorphism. Figure~\ref{fig:cardiac}
shows the result, and that accelerated gradient descent is able to
register quite well the two images. This shows that accelerated
gradient descent can capture more general diffeomorphisms.

\begin{figure}
  \label{fig:cardiac}
\end{figure}
}

\section{Conclusion}
We have generalized accelerated optimization, in particular Nesterov's
scheme, to infinite dimensional manifolds. This method is general and
applies to optimizing any functional on an infinite dimensional
manifold.  We have demonstrated this for the class of diffeomorphisms
motivated by variational optical flow problems in computer vision. The
main objective of the paper was to introduce the formalism and derive
the evolution equations that are PDEs. The evolution equations are
natural extensions of mechanical principles from fluid mechanics, and
in particular connect to optimal mass transport. They require
additional evolution equations over gradient descent, i.e., a velocity
evolution and a density evolution, but that does not significantly add
to the cost of $L^2$ gradient descent per iteration since the updates
are all local, i.e., computation of derivatives. Our numerical scheme
to implement these equations used entropy conditions, which were
employed to cope with shocks and fans of the underlying PDE. Our
numerical scheme is not final and could be improved, and we plan to
explore this in future work. Experiments on toy examples nevertheless
demonstrated the advantages of speed and robustness to local minima
over gradient descent, and illustrated the behavior of accelerated
gradient descent. Just by simple acceleration, gradient descent,
unusable in practice due to scalability with image size, became
usable. One area that should be explored further is the choice of the
time-explicit functions $a,b$ in the generalized Lagrangian. These
were chosen to coincide with the choices to produce the continuum
limit of Nesterov's scheme finite dimensions, which are designed for
the convex case to yield optimal convergence. Since the energies that
we consider are non-convex, these may no longer be optimal. Of
interest would be a design principle for choosing $a,b$ so as to
obtain optimal convergence rates. A follow-up question would then be
whether the discretization of the PDEs gives optimal rates in
discrete-time.  Another issue is that we assumed that the domain of
the diffeomorphism was $\R^n$, but images are compact; we by-passed
this complication by assuming periodic boundary conditions. Future
work will look into proper treatment of the boundary.

\appendix

\subsection{Functional Gradients}
\label{app:funct_grads}

\begin{definition}[Functional Gradients]
  Let $U : \diff \to \R$. The gradient (or functional derivative) of
  $U$ with respect to $\phi \in \diff$, denoted $\nabla U(\phi)$, is
  defined as the $\nabla U(\phi) \in T_{\phi} \diff$ that satisfies
  \begin{equation}
    \delta U(\phi) \cdot v = \int_{\phi(\R^n)} \nabla U(\phi)(x) \cdot
    v(x) \ud x
  \end{equation}
  for all $v\in T_{\phi} \diff$. The left hand side is the directional
  derivative and is defined as 
  \begin{equation}
    \delta U(\phi) \cdot v := \left. \der{}{\varepsilon} U( \phi +
      \varepsilon v ) \right|_{\varepsilon= 0}.
  \end{equation}
  Note that $(\phi + \varepsilon v)(x) = \phi(x) + \varepsilon v(
  \phi(x) )$ for $x\in \R^n$.
\end{definition}

We now show the computation of the gradient for the illustrative
potential \eqref{eq:potential_illustrative} used in this paper. First,
let us consider the data term $U_1(\phi) = \int_{\R^n} |I_1(\phi(x)) -
I_0(x)|^2 \ud x$ then 
\[
  \delta U_1(\phi) \cdot \delta\phi = \int_{\R^n} 2(I_1(\phi(x)) -
  I_0(x)) DI_1(\phi(x)) \widehat{\delta \phi}(x) \ud x = 
  \int_{\phi(\R^n)} 2( I_1(x) - I_0(\psi(x)) ) DI_1(x) \delta \phi(x)
  \det{\nabla\psi(x)}\ud x,
\]
where $\widehat{\delta\phi} = \delta\phi\circ\phi$, $\psi = \phi^{-1}$
and we have performed a change of variables. Thus,
$\nabla U_1 = 2\nabla I_1( I_1 - I_0\circ\psi) \det{\nabla\psi}$. Now
consider the regularity term
$U_2(\phi) = \int_{\R^n} |\nabla (\phi(x)-x)|^2 \ud x$, then
\[
  \delta U(\phi) = 2\int_{\R^n} \tr{ \nabla (\phi(x)-\mbox{id})^T \nabla
                   \widehat{\delta\phi}(x)    } \ud x =
                 -\int_{\R^n} \Delta \phi(x)^T
    \delta\phi(x) \ud x =
    \int_{\Omega} (\Delta \phi)(\psi(x))^T  \delta \phi(x)
    \det{ \nabla\psi(x) }\ud x.
\]
Note that in integration by parts, the boundary term vanishes since we
assume that $\phi(x) = x$ as $|x|\to \infty$. Thus, $\nabla U_2 =
(\Delta \phi) \circ \psi  \det{ \nabla\psi }$.

\subsection{Stationary Conditions}
\label{app:stat_cond_nondissip}

\begin{lemma}[Stationary Condition for the Mapping]
  \label{eq:stationary_mapping}
  The stationary condition of the action \eqref{eq:action_lagrange}
  for the mapping is
  \begin{equation}
    \partial_t \lambda + \dv{v\lambda^T}^T = (\nabla\psi)^{-1} \nabla U(\phi).
  \end{equation}
\end{lemma}
\begin{proof}
  We compute the variation of $A$ (defined in
  \eqref{eq:action_lagrange}) with respect to the mapping $\phi$. The
  only terms in the action that depend on the mapping are $U$ and the
  Lagrange multiplier term associated with the mapping. Taking the
  variation w.r.t the potenial term gives
  \[
    -\int \int_{ \phi(\R^n) } \nabla U(\phi) \cdot \delta \phi \ud x
    \ud t.
  \]
  Now the variation with respect to the Lagrange multiplier term:
  \[
    \int\int_{ \phi(\R^n) } \lambda^T[ \partial_t \widehat{\delta \psi} + D
    (\widehat{\delta\psi}) v ] \ud x \ud t
    = -\int \int_{ \phi(\R^n) } [ \partial_t \lambda^T +
    \dv{v\lambda^T} ] \widehat{\delta\psi} \ud x \ud t,
  \]
  where we have integrated by parts, the $\dv{\cdot}$ of a matrix means
  the divergence of each of the columns, resulting in a row
  vector, and $\widehat{\delta \psi} = \delta \psi \circ\psi$. Note
  that we can take the variation of $\psi(\phi(x)) = x$ to obtain
  \[
    \widehat{\delta\psi} \circ \phi(x) + [D\psi(\phi(x))] \widehat{\delta\phi}(x) = 0,
  \]
  or 
  \[
    \widehat{\delta\psi}(y) = -[D\psi(y)] \delta\phi(y).
  \]
  Therefore,
  \begin{equation} \label{eq:der_action_mapping}
    \delta A \cdot \delta\phi = 
    \int\int_{ \phi(\R^n) } \left\{ (\nabla \psi) \left[  \partial_t \lambda +
        \dv{v\lambda^T}^T \right] - \nabla U(\phi) \right\} \cdot \delta \phi
    \ud x \ud t.
  \end{equation}
\end{proof}

\begin{lemma}[Stationary Condition for the Velocity]
  The stationary condition of the action \eqref{eq:action_lagrange}
  arising from the velocity is
  \begin{equation}
    \rho v + (\nabla\psi)\lambda - \rho\nabla\mu = 0.
  \end{equation}
\end{lemma}
\begin{proof}
  We compute the variation w.r.t the kinetic energy:
  \[
    \delta T \cdot \delta v = \int_{\phi(\R^n)} \rho v\cdot \delta v
    \ud x. 
  \]
  The variation of the Lagrange multiplier terms is 
  \[
    \int\int_{\phi(\R^n)} \lambda^T (D\psi)  \delta v - \rho \nabla\mu
    \cdot \delta v \ud x \ud t = 
    \int\int_{\phi(\R^n)} [ (\nabla\psi)\lambda - \rho\nabla\mu ]
    \cdot \delta v \ud x \ud t.
  \]
  Therefore,
  \begin{equation} \label{eq:var_action_vel}
    \delta A \cdot \delta v = 
    \int\int_{\phi(\R^n)} [ \rho v + (\nabla\psi)\lambda - \rho\nabla\mu ]
    \cdot \delta v \ud x \ud t.
  \end{equation}
\end{proof}

\begin{lemma}[Stationary Condition for the Density]
  The stationary condition of the action \eqref{eq:action_lagrange}
  arising from the velocity is
  \begin{equation}
    \partial_t \mu + (D\mu)v = \frac 1 2 |v|^2.
  \end{equation}
\end{lemma}
\begin{proof}
  Note that the terms that contain the density in
  \eqref{eq:action_lagrange} are the kinetic energy and the Lagrange
  multiplier corresponding to the density. We see that
  \begin{equation} \label{eq:var_action_rho}
    \delta A \cdot \delta\rho = \int\int_{\phi(\R^n)} \frac 1 2 |v|^2
    \delta\rho - (\partial_t\mu + \nabla\mu\cdot v) \delta \rho \ud
    x\ud t,
  \end{equation}
  which yields the lemma.
\end{proof}

\subsection{Velocity Evolution}
\label{app:velocity_evol}

\begin{lemma}\label{lem:lambda_t_w_t}
  Given that $(\nabla\psi)\lambda = w$, we have that
  \begin{equation} 
    \partial_t\lambda +  (D\lambda)v + \lambda \dv{v} = 
    (\nabla\psi)^{-1} [ \partial_t w + (Dw)v + (\nabla v)w + w\dv{v} ]    
  \end{equation}
\end{lemma}
\begin{proof}
  Define the Hessian as follows:
  \[
    [ D^2\psi ]_{ijk} = \partial_{x_ix_j}^2 \psi^k, \quad 
    [ D^2\psi(a,b) ]_k =  \sum_{ij} \partial_{x_ix_j}^2 \psi^k a_i b_j.
  \]
  We compute 
  \[
    \{ D[ (\nabla\psi)\lambda ] \}_{ ij  } = 
    \partial_{ x_j  } [ (\nabla\psi)\lambda ]_{i}  = 
    \partial_{ x_j  } \sum_{ l } \partial_{ x_i } \psi^l \lambda_l = 
    \sum_l (\partial^2_{ x_j x_i } \psi^l \lambda_l) + \partial_{ x_i }
    \psi^l\partial_{ x_j } \lambda_l.
  \]
  Therefore,
  \[
    D[ (\nabla\psi)\lambda ]= D^2\psi(\cdot, \cdot) \cdot \lambda + (\nabla\psi)  (D\lambda)
  \]
  Since $ D[ (\nabla\psi)\lambda ] = Dw $ then solving for $D\lambda$ gives
  \[
    D\lambda =(\nabla\psi)^{-1} [ Dw - D^2\psi(\cdot, \cdot) \cdot
    \lambda ],
  \]
  so 
  \begin{equation} \label{eq:Dlam_v}
    (D\lambda)v = 
    (\nabla\psi)^{-1} [ (Dw)v - D^2\psi(\cdot, v) \cdot
    \lambda ].
  \end{equation}

  Now differentiating $(\nabla\psi)\lambda = w$ w.r.t $t$, we have
  \[
    (\nabla \partial_t\psi)\lambda + (\nabla\psi)\partial_t \lambda
    = \partial_t w, \quad \mbox{or} \quad
    \partial_t \lambda = (\nabla \psi)^{-1} [ \partial_t w - (\nabla \partial_t\psi)\lambda] 
  \]
  Note that $\partial_t\psi = -(D\psi)v$ so 
  \begin{equation} \label{eq:lambda_t_w}
    \partial_t \lambda = (\nabla \psi)^{-1} \left\{ \partial_t w +
      \nabla [ (D\psi)v] \lambda\right\}.
  \end{equation}
  Now computing $\nabla [ (D\psi)v]$ yields
  \[
    \{ \nabla [ (D\psi)v ) ]  \}_{lk} = 
  \partial_{ x_l } \sum_{ i } \partial_{ x_i } \psi^k v^i = 
  \sum_{i} \partial_{x_l}\partial_{x_i} \psi^k v^i + 
  \partial_{x_i} \psi^k \partial_{x_l}  v^i.
  \]
  Then multiplying the above matrix by $\lambda$ gives
  \[
    \{ \nabla [ (D\psi)v ) ] \lambda \}_{l} = \sum_{ik} 
    \partial_{x_l}\partial_{x_i} \psi^k v^i \lambda^k + 
    \partial_{x_i} \psi^k \partial_{x_l}  v^i \lambda^k,
  \]
  which in matrix form is
  \[
    \nabla [ (D\psi)v ) ] \lambda = 
    D^2\psi( \cdot, v )\cdot \lambda + (\nabla v)(\nabla\psi)\lambda = 
    D^2\psi( \cdot, v )\cdot \lambda + (\nabla v)w
  \]
  Therefore, \eqref{eq:lambda_t_w} becomes
  \[
    \partial_t\lambda = (\nabla\psi)^{-1} [ 
    \partial_t w + D^2\psi( \cdot, v )\cdot \lambda + (\nabla v)w
    ].
  \]
  Combining the previous with \eqref{eq:Dlam_v} and noting that
  $\lambda \dv{v} = (\nabla\psi)^{-1}w \dv{v}$ yields
  \[
    \partial_t\lambda +  (D\lambda)v + \lambda \dv{v} = 
    (\nabla\psi)^{-1} [ \partial_t w + (Dw)v + (\nabla v)w + w\dv{v} ].    
  \]
\end{proof}

\begin{lemma} \label{lem:w_t_v_t}
  If $w = \rho( \nabla\mu - v )$, then 
  \begin{equation}
    \partial_t w + (Dw)v + (\nabla v)w + w\dv{v} = -\rho[ \partial_t v
    + (Dv)v ].
  \end{equation}
\end{lemma}
\begin{proof}
  Differentiating $w = \rho( \nabla\mu - v )$, we have
  \begin{align*}
    \partial_t w &= (\partial_t\rho)(\nabla\mu - v) + \rho(
                   \nabla\partial_t\mu - \partial_t v) \\
    Dw &= (\nabla\mu - v)(D\rho) + \rho[ D(\nabla\mu) - Dv ].
  \end{align*}
  Therefore,
  \begin{align*}
    \partial_t w + (Dw)v + (\nabla v)w + w\dv{v} &= 
    (\nabla\mu - v)( \partial_t\rho +\nabla\rho\cdot v ) + 
    \rho[ \nabla\partial_t\mu - \partial_t v + D(\nabla\mu)v - (Dv)v
    ] \\
    &+ 
    \rho(\nabla v)(\nabla\mu - v) + \rho(\nabla\mu - v)\dv{v} \\
    &= 
      (\nabla\mu - v)( \partial_t\rho +\nabla\rho\cdot v + \rho\dv{v}
      ) \\
    &+ 
    \rho[ \nabla\partial_t\mu - \partial_t v + D(\nabla\mu)v - (Dv)v
      + (\nabla v)(\nabla\mu - v)
    ].
  \end{align*}
  Note that $\partial_t\rho +\nabla\rho\cdot v +
  \rho\dv{v} = \partial_t \rho + \dv{\rho v} = 0$, due to the
  continuity equation. Therefore,
  \begin{align*}
    \partial_t w + (Dw)v + (\nabla v)w + w\dv{v} &=
    \rho[ -\partial_t v - (Dv)v - (\nabla v)v + 
    \nabla\partial_t\mu + D(\nabla\mu)v + (\nabla v)(\nabla\mu) ] \\
    &=
      \rho\left\{ -\partial_t v - (Dv)v - (\nabla v)v + 
      \nabla[  \partial_t\mu + (D\mu)v ] \right\}.
  \end{align*}
  By the stationary condition for the density, $\partial_t\mu +
  (D\mu)v = 1/2 |v|^2$, so $\nabla[ \partial_t\mu +
  (D\mu)v ] = (\nabla v)v$, which gives the lemma.
\end{proof}

\begin{theorem}[Velocity Evolution]
  The evolution equation for the velocity arising from the
  stationarity of the action integral is 
  \begin{equation}
    \rho[ \partial_t v + (Dv)v ] = -\nabla U(\phi).
  \end{equation}
\end{theorem}
\begin{proof}
  This is a combination of  Lemmas~\ref{eq:stationary_mapping},
 \ref{lem:lambda_t_w_t}, and \ref{lem:w_t_v_t}.  
\end{proof}

\subsection{Stationary Conditions for the Dissipative Case}

\label{app:stationary_dissip}

\begin{theorem}[Stationary Conditions for the Path of Least Action:
  Dissipative Case] \label{thrm:stationary_lagrange_mult_dissip}
  The stationary conditions of the path for the action
  \begin{align}
    \label{eq:action_dissip_lagrange}
  A &= \int \left[ a T(v) - b U(\phi) \right] \ud t
      + \int \int_{\R^n} \lambda^T[ \partial_t \psi_t + (D\psi)v ] \ud
      x \ud t
      -\int \int_{\R^n} \left[ \partial_t \mu + \nabla \mu \cdot
      v \right] \rho \ud x \ud t,
  \end{align}
  are
  \begin{align}
    \partial_t \lambda + (D\lambda )v + \lambda \dv{v}  &= b(\nabla\psi)^{-1}\nabla
                                                          U(\phi) \\
    a\rho v + (\nabla \psi) \lambda - \rho\nabla\mu &= 0 \\
    \partial_t \mu + \nabla\mu \cdot v &= \frac 1 2 a|v|^2.
  \end{align}
\end{theorem}
\begin{proof}
  Note that
  \begin{align*}
    \nabla[ b U ](\phi) &= b\nabla U(\phi)\\  
    \delta [ aT ] \cdot \delta \rho &= \int_{\phi(\R^n)} \frac 1 2 a |v|^2\delta \rho
                                      \ud x\\
    \delta [aT] \cdot \delta v &=  \int_{\phi(\R^n)}
                                 a \rho v\cdot\delta v \ud x.
  \end{align*}
  Therefore, using \eqref{eq:der_action_mapping} and replacing $\nabla
  U(\phi)$ with $b\nabla U(\phi)$, we have 
  \[
    \delta A \cdot \delta\phi = 
    \int\int_{ \phi(\R^n) } \left\{ (\nabla \psi) \left[  \partial_t \lambda +
        \dv{v\lambda^T}^T \right] - b\nabla U(\phi) \right\} \cdot \delta \phi
    \ud x \ud t,
  \]
  which yields the stationary condition on the mapping. Also, updating
  \eqref{eq:var_action_vel} yields
  \[
    \delta A \cdot \delta v = 
    \int\int_{\phi(\R^n)} [ a \rho v + (\nabla\psi)\lambda - \rho\nabla\mu ]
    \cdot \delta v \ud x \ud t,
  \]
  which yields the stationary condition for the velocity. Finally,
  updating \eqref{eq:var_action_rho} yields
  \[
    \delta A \cdot \delta\rho = \int\int_{\phi(\R^n)} \frac 1 2 a|v|^2
    \delta\rho - (\partial_t\mu + \nabla\mu\cdot v) \delta \rho \ud
    x\ud t,
  \]
  and that yields the last stationary condition.
\end{proof}

\begin{theorem}[Evolution Equations for the Path of Least Action:
  Dissipative Case]
  The evolution equations for the stationary conditions of the action
  in \eqref{eq:action_dissip_lagrange} is 
  \begin{equation}
    \rho[ \partial_t(av) + a(Dv)v ] = -b\nabla U(\phi).
  \end{equation}
\end{theorem}
\begin{proof}
  Let $w = \rho( \nabla\mu-av)$ then 
  \begin{align*}
    \partial_t w &= (\partial_t\rho)(\nabla\mu - av) + \rho(
                   \nabla\partial_t\mu - \partial_t (av) ) \\
    Dw &= (\nabla\mu - av)(D\rho) + \rho[ D(\nabla\mu) - aDv ].
  \end{align*}
  Then
  \begin{align*}
    \partial_t w + (Dw)v + (\nabla v)w + w\dv{v} &= g
    (\nabla\mu - av)( \partial_t\rho +\nabla\rho\cdot v ) + 
    \rho[ \nabla\partial_t\mu - \partial_t (av) + D(\nabla\mu)v - a(Dv)v
    ] \\
    &+ 
    \rho(\nabla v)(\nabla\mu - av) + \rho(\nabla\mu - av)\dv{v} \\
    &= 
      (\nabla\mu - av)( \partial_t\rho +\nabla\rho\cdot v + \rho\dv{v}
      ) \\
    &+ 
    \rho[ \nabla\partial_t\mu - \partial_t (av) + D(\nabla\mu)v - a(Dv)v
      + (\nabla v)(\nabla\mu - av)
    ] \\
    &= \rho\left\{ -\partial_t(av) - a(Dv)v - a(\nabla v)v +
      \nabla[ \partial_t\mu + (D\mu)v  ] \right\} \\
    &= \rho\left\{ -\partial_t(av) - a(Dv)v \right\}.
  \end{align*}
  By Lemma~\ref{lem:lambda_t_w_t} and the previous expression, we have
  our result.
\end{proof}

\subsection{Discretization}\label{app:discretization}
We present the discretization of the velocity PDE
\eqref{eq:vel_evol_dissp_nesterov} first.  In one dimension, the terms
involving $v$ are Burger's equation, which is known to produce
shocks. We thus use an entropy scheme. Writing the PDE component-wise,
we get 
\begin{align} \label{eq:velocity_ev_comp}
  \partial_t v_1 &= -\frac 1 2 \partial_{x_1} (v_1)^2 -
                   v_2 \partial_{x_2} v_1 - \frac 3 t v_1 - \frac{1}{\rho} (\nabla U)_1 \\
  \partial_t v_2 &= -\frac 1 2 \partial_{x_2} (v_2)^2 -
                   v_1 \partial_{x_1} v_1 - \frac 3 t v_2 - \frac {1}{\rho} (\nabla U)_2,
\end{align}
where the subscript indicates the component of the vector. We use
forward Euler for the time derivative, and for the first term on the
right hand side, we use an entropy scheme for Burger's equation which
results in the following discretization:
\[
  \partial_{x_1} (v_1)^2(x) \approx
  \max\{ v_1(x), 0 \}^2 - \min\{ v_1(x), 0 \}^2 + 
  \min\{ v_1(x_1+\Delta x, x_2), 0 \}^2 - \max\{ v_1(x_1+\Delta x,
  x_2), 0 \}^2,
\]
where $\Delta x$ is the spatial sampling size, and the
$\partial_{x_2} (v_2)^2$ follows similarly. For the second term on the
right hand side of \eqref{eq:velocity_ev_comp}, we follow the
discretization of a transport equation using an up-winding scheme,
which yields the following discretization:
\[
  v_2(x) \partial_{x_2} v_1(x) \approx v_2(x) \cdot 
  \begin{cases}
    v_1( x_1, x_2 ) - v_1( x_1, x_2 -\Delta x ) & v_2(x) > 0 \\
    v_1( x_1, x_2 +\Delta x ) - v_1( x_1, x_2 ) & v_2(x) < 0 
  \end{cases}.
\]
With regards to the gradient of potential, if we use the potential
\eqref{eq:potential_illustrative}, then all the derivatives are
discretized using central differences, as the key term is a
diffusion.  The step size $\Delta t / \Delta x < 1/\max_x\{ |v(x)|,
|Dv(x)| \}$.

The backward map $\psi$ evolves according to a transport PDE
\eqref{eq:transport_backward}, and thus an up-winding scheme similar to
the transport term in the velocity term is used. For the
discretization of the continuity equation, we use a staggered grid (so
that the values of $v$ are defined in between grid points and $\rho$
is defined at the grid points). The discretization is just the sum of
the fluxes coming into the point:
\[
  -\dv{\rho(x)v(x)} \approx
  \sum_{i=1}^2 \left[
  -v_i(x)
  \begin{cases}
    \rho(x) & v_i(x) > 0 \\
    \rho(x +\Delta x_i) & v_i(x) <0
  \end{cases} + 
  v_i(x-\Delta x_i) 
  \begin{cases}
    \rho(x-\Delta x_i) & v_1(x-\Delta x_i) > 0 \\
    \rho(x) & v_1(x-\Delta x_i) < 0
  \end{cases}
  \right],
\]
where $\Delta x_i$ denotes the vector of the spatial increment
$\Delta x$ in the $i^{\text{th}}$ coordinate direction, $v_1(x)$
denotes the velocity defined at the midpoint between $(x_1,x_2)$ and
$(x_1+\Delta x, x_2)$, and $v_2(x) $ denotes the velocity defined at
the midpoint between $(x_1,x_2)$ and $(x_1, x_2+\Delta x)$. The term
$\partial_t\rho(x)$ is discretized with forward Euler. This scheme is
guaranteed to preserve mass.

\bibliographystyle{IEEEtran}
\bibliography{accel}

\end{document}